\documentclass[11pt,a4paper]{amsart}

\usepackage[utf8]{inputenc}
\usepackage[english]{babel}
\usepackage{amsmath}
\usepackage{amssymb}
\usepackage{amsfonts} 
\usepackage{amsthm}
\usepackage{mathrsfs}
\usepackage{graphicx}
\usepackage{enumitem}
\setlist[itemize]{leftmargin=*}
\setlist{itemsep=4pt plus 2pt minus 1pt,topsep=4pt plus 2pt minus 1pt}
\usepackage{xcolor}
\usepackage{url}
\usepackage{accents}
\usepackage{bm}
\usepackage{hyperref,doi}
\usepackage{mathtools}
\usepackage{etoolbox}
\usepackage{comment}
\usepackage{chngcntr}
\usepackage[a4paper,left=3cm,right=3cm,top=3.4cm,bottom=3.4cm]{geometry}
\usepackage{tikz,tikz-cd}
\usepackage{appendix}
\usepackage{scalerel}
\usepackage[usestackEOL]{stackengine}
\usepackage[stretch=30,shrink=30]{microtype}
\usepackage[strings]{underscore}
\usepackage[square,numbers]{natbib}

\makeatletter
\newcommand{\boxedred}[1]{\textcolor{red}{\fbox{\normalcolor\m@th$#1$}}}
\makeatother
\makeatletter
\newcommand{\eqlabel}[1]{\refstepcounter{equation}\textup{\tagform@{\theequation}}\label{#1}}
\makeatother

\newcommand{\N}{\mathbb{N}}
\newcommand{\Z}{\mathbb{Z}}

\newcommand{\R}{\mathbb{R}}
\newcommand{\C}{\mathbb{C}}

\newcommand{\floor}[1]{\left\lfloor #1\right\rfloor}
\newcommand\restr[2]{{
		\left.\kern-\nulldelimiterspace
		#1
		\vphantom{\big|}
		\right|_{#2}
}}

\newcommand{\de}{\partial}
\newcommand{\co}[1]{\operatorname{co}\left(#1\right)}

\newcommand{\mz}{\frac{1}{2}}

\newcommand{\ang}[1]{\left\langle#1\right\rangle}
\newcommand{\uno}{\bm{1}}
\newcommand{\supp}[1]{\operatorname{supp}\left(#1\right)}
\newcommand{\nin}{\not\in}

\newcommand{\weakto}{\rightharpoonup}
\newcommand{\weakstarto}{\stackrel{*}{\rightharpoonup}}

\newcommand{\cptsub}{\subset\hspace{-1pt}\subset}

\DeclareMathOperator{\diam}{diam}

\DeclareMathOperator{\dist}{dist}

\newtheoremstyle{plainstyle}
{\glueexpr\parskip*4\relax}                    
{\glueexpr\parskip*2\relax}                    
{}                   
{}                           
{\bfseries}                   
{.}                          
{.2em}                       
{\thmname{#1}\thmnumber{ #2}\thmnote{ (#3)}}  
\theoremstyle{plainstyle}

\newtheorem{definition}{Definition}[section]
\newtheorem{rmk}[definition]{Remark}

\newtheoremstyle{itstyle}
{\glueexpr\parskip*4\relax}                    
{\glueexpr\parskip*2\relax}                    
{\itshape}                   
{}                           
{\bfseries}                   
{.}                          
{.2em}                       
{\thmname{#1}\thmnumber{ #2}\thmnote{ (#3)}}  
\theoremstyle{itstyle}

\newtheorem{thm}[definition]{Theorem}
\newtheorem{lemmaen}[definition]{Lemma}
\newtheorem{corollary}[definition]{Corollary}

\newtheorem{conj}[definition]{Conjecture}
\newtheorem*{thm*}{Theorem}
\newtheorem*{corollary*}{Corollary}

\newcommand{\obar}[1]{\overline{#1}}

\newcommand{\set}[1]{\left\{#1\right\}}

\newcommand{\pa}[1]{\left(#1\right)}

\newcommand{\abs}[1]{\left|#1\right|}

\newcommand{\norm}[1]{\left\|#1\right\|}

\newcommand{\brapa}[1]{\left[#1\right)}

\newcommand{\vfd}{\mathbf{v}}
\newcommand{\envdim}{q}
\newcommand{\subman}{\mathcal{M}^m}

\newcommand{\good}{\mathcal{G}}

\newcommand{\bad}{\mathcal{B}}

\newcommand{\e}{\widetilde{e}}

\renewcommand{\bar}{\obar}
\renewcommand{\hat}{\widehat}
\renewcommand{\tilde}{\widetilde}
\renewcommand{\epsilon}{\varepsilon}

\def\media{
	\,\ThisStyle{\ensurestackMath{%
			\stackinset{c}{.2\LMpt}{c}{.5\LMpt}{\SavedStyle-}{\SavedStyle\phantom{\int}}}%
		\setbox0=\hbox{$\SavedStyle\int\,$}\kern-\wd0}\int
}

\counterwithin*{equation}{section}
\renewcommand{\theequation}{\thesection.\arabic{equation}}

\subjclass[2010]{49Q05, 49Q15, 49Q20, 58E20}
\keywords{Minimal surfaces, min-max, viscosity method, multiplicity one conjecture}

\begin{document}
	
	\raggedbottom

	\author[A.~Pigati]{Alessandro Pigati}

	\author[T.~Rivi{\`e}re]{Tristan Rivi{\`e}re}

	\thanks{(Pigati and Rivi{\`e}re) {\sc ETH Z\"urich, Department of Mathematics,
		R\"amistrasse 101, 8092 Z\"urich, Switzerland}. E-mail addresses: \href{mailto:alessandro.pigati@math.ethz.ch}{alessandro.pigati@math.ethz.ch} and \href{mailto:tristan.riviere@math.ethz.ch}{tristan.riviere@math.ethz.ch}.}

	\title[Multiplicity one for min-max minimal surfaces in arbitrary codimension]{A proof of the multiplicity one conjecture for min-max minimal surfaces in arbitrary codimension}

	\begin{abstract}
		Given any admissible $k$-dimensional family of immersions of a given closed oriented surface into an arbitrary closed Riemannian manifold, we prove that the corresponding min-max width for the area
		is achieved by a smooth (possibly branched) immersed minimal surface with multiplicity one and Morse index bounded by $k$.
	\end{abstract}

	\maketitle

	\section{Introduction}

	
	Recently, a new theory for the construction of branched immersed minimal surfaces of arbitrary topology, in an assigned closed Riemannian manifold $\subman$, was proposed in \cite{rivminmax}. This method is based on a penalization of the area functional by means of the second fundamental form $A$ of the immersion.
	
	Namely, for a fixed parameter $\sigma>0$, one first finds an immersion $\Phi:\Sigma\to\subman$ which is critical for the perturbed area functional
	\begin{align}\label{pertfct} A^{\sigma}(\Phi):=\int_\Sigma\,d\operatorname{vol}_{g_\Phi}+\sigma^2\int_\Sigma(1+\abs{A}_{g_\Phi}^2)^2\,d\operatorname{vol}_{g_\Phi}, \end{align}
	where $\Sigma$ is a fixed closed oriented surface and $g_\Phi$ is the metric induced by $\Phi$, with volume form $\operatorname{vol}_{g_\Phi}$. This functional $A^\sigma$ enjoys a sort of Palais--Smale condition up to diffeomorphisms.
	
	We should mention that the idea of considering perturbed functionals goes back to the paper \cite{sacks} by Sacks--Uhlenbeck, where a perturbation of the Dirichlet energy is used to build minimal immersed spheres. However, in order to find minimal immersed surfaces with higher genus, one should give up working with the Dirichlet energy and use a more tensorial functional like \eqref{pertfct}: among closed orientable surfaces, only the sphere has a unique conformal structure (up to diffeomorphisms) and, as a consequence, a harmonic map (i.e. a critical point for the Dirichlet energy) $\Phi:\Sigma\to\subman$ could fail to be conformal and minimal if $\Sigma$ has positive genus. In principle, one can overcome this issue by introducing the conformal structure as an additional parameter in the variational problem: this program was carried out by Zhou in \cite{zhou}.
	
	Considering any sequence $\sigma_j\downarrow 0$, one gets a sequence $\Phi_j:\Sigma_j\to M$ of conformal immersions (with area bounded above and below), where $\Sigma_j$ denotes $\Sigma$ endowed with the induced conformal structure. Assuming for simplicity that we are dealing with a constant conformal structure (in general one gets a limiting Riemann surface in the Deligne--Mumford compactification), the sequence $\Phi_j$ is then bounded in $W^{1,2}$ and we can consider its weak limit $\Phi_\infty$, up to subsequences. A priori it is not clear whether the strong $W^{1,2}$-convergence holds, even away from a finite bubbling set. However, in \cite{rivminmax} the second author shows that, if the sequence $\sigma_j$ is carefully chosen so as to satisfy a certain \emph{entropy condition}, then the surfaces $\Phi_j(\Sigma_j)$ converge to a \emph{parametrized stationary varifold} (a notion introduced in \cite{rivminmax,pigriv} and recalled in Section \ref{backsec} below) which we call $(\Sigma_\infty,\Theta_\infty,N_\infty)$ in the present paper. The limiting multiplicity $N_\infty$ a priori could be bigger than one.
	
	
	A consequence of the main regularity result contained in \cite{pigriv} is that the multiplicity $N_\infty$ is locally constant.
	This result, which is optimal for the class of parametrized stationary varifolds, leaves nonetheless open the question whether one can have $N_\infty>1$ on some connected component of $\Sigma_\infty$.


	This question should be compared with the \emph{multiplicity one conjecture} by Marques and Neves. In \cite{marnevmult}, the following upper bound for the Morse index of a minimal hypersurface with locally constant multiplicity is established: if
	\begin{align*} \Sigma=\sum_{j=1}^\ell n_j\Sigma_j \end{align*}
	is a minimal hypersurface with locally constant multiplicity, given by a min-max with $k$ parameters in the context of Almgren--Pitts theory, then
	\begin{align*} \operatorname{index}(\supp{\Sigma})\le k,\qquad\supp{\Sigma}:=\bigsqcup_{j=1}^\ell\Sigma_j. \end{align*}
	
	In other words, this is a bound for the Morse index of the hypersurface obtained by replacing all the multiplicities $n_j$ with $1$.
	In order for this estimate to give more information about $\Sigma$, or at least its unstable part, the authors make the following conjecture.
	
	\begin{conj}[Multiplicity one conjecture] For generic metrics on $M^{n+1}$, with $3\le n+1\le 7$, two-sided unstable components of closed minimal hypersurfaces obtained by min-max methods must have multiplicity one. \end{conj}
	
	It is natural to demand for extra information for one-sided stable components with unstable double cover, as well.	
	Marques and Neves were able to prove this conjecture for one-parameter sweepouts, leaving the general case open. For metrics with positive Ricci curvature, related results were already obtained by Marques and Neves in \cite{marnevrig} and later by Zhou in \cite{zhouric}.
	
	Further results, such as the two-sidedness of $\Sigma$ when the metric has positive Ricci curvature, were obtained by Ketover, Marques and Neves in \cite{ketmarnev}, using the catenoid estimate.
	
	We also mention that Ketover, Liokumovich and Song in \cite{ketlio,song} started to settle the generic, one-parameter case for the simpler and more effective Simon--Smith variant of Almgren--Pitts theory, specially designed for 3-manifolds.
	
	Very recently, in \cite{chma}, Chodosh and Mantoulidis established the conjecture for bumpy metrics in 3-manifolds, i.e. when $n=2$, in the setting of Allen--Cahn level set approach. Also, half a year after this paper was written, Zhou announced a proof for bumpy metrics in all dimensions $3\le n+1\le 7$, in the context of Almgren--Pitts theory (see \cite{zhoubumpy}).
	
	The importance of this conjecture in relation to the Morse index of $\Sigma$ is twofold. First of all, there is no satisfactory definition for the Morse index of an embedded minimal hypersurface with multiplicity bigger than one: such $\Sigma$ could be thought as the limiting object of many qualitatively different sequences, e.g. the elements of the sequence could realize different covering spaces of the limit, or more pathologically they could be minimal and have many catenoidal necks (hence $\Sigma$ would be the limit of a sequence of highly unstable hypersurfaces).
	
	Also, if one is able to establish a lower bound on the Morse index such as
	\begin{align*} k\le\operatorname{index}(\supp{\Sigma})+\operatorname{nullity}(\supp{\Sigma}), \end{align*}
	then the multiplicity one conjecture gives infinitely many \emph{geometrically distinct} minimal hypersurfaces, provided there exists at least one for every value of $k$. This was precisely the strategy used in \cite{chma} to prove Yau's conjecture for generic metrics: in \cite{chma} the authors obtained the multiplicity one result and the equality $\operatorname{index}(\Sigma)=k$ (the nullity vanishing automatically for bumpy metrics). This was later extended to higher ambient dimension (but in codimension one) in \cite{zhoubumpy}.
	
	In this work we establish the natural counterpart of this conjecture in our setting, namely for minimal surfaces produced by the viscous relaxation method.
	
	\begin{thm}\label{main}
		We have $N_\infty\equiv 1$.
	\end{thm}

	We stress that this result holds in \emph{arbitrary codimension} and \emph{without any genericity assumption}.
	This should be seen as a multiplicity one statement from the perspective of the parametrization \emph{domain},
	in that localization in the domain (away from branch points) gives a genuine embedded minimal surface, but a priori it
	does not exclude multiple covers of the image surface $\Theta_\infty(\Sigma_\infty)$ globally.
	It seems to be optimal for a min-max approach involving parametrizations, rather than e.g. approaches involving level sets of functions, and it is sufficient to obtain an upper bound on the Morse index. This bound, detailed in \cite{rivlower}, relies on having a branched immersion at our disposal, for which a good definition of Morse index is available.
	
	We remark that, in view of earlier work in \cite{rivtarget}, Theorem \ref{main} would imply by itself the main result of \cite{pigriv}, for parametrized stationary varifolds arising as a limit of stationary points for the relaxed functionals. However, the proof of Theorem \ref{main} relies substantially on the regularity result obtained in \cite{pigriv}, needed in several compactness arguments.
	
	The main idea is to define a sort of \emph{macroscopic multiplicity}, on balls $B_\ell^q(p)$ in an ambient Euclidean space $\R^q\supseteq\subman$, before passing to the limit (i.e. looking at the immersed surfaces $\Phi_j$ rather than their limit).
	This macroscopic multiplicity is roughly the closest integer to the average of a \emph{projected multiplicity},
	issued by the map $\Pi\circ\restr{\Phi_j}{B_r^2(z)}$, where $B_r^2(z)$ is a small domain ball and $\Pi$ a 2-plane close to the image of $\restr{\Phi_j}{\de B_r^2(z)}$. Then we will use a continuity argument to show that this number stays constant as we pass from scale $1$ to scale $\sqrt{\sigma_j}$. At the latter scale we have a very clear understanding of the behaviour of $\Phi_j$ and in particular we are able to say that here the macroscopic multiplicity equals $1$. Thus the same holds at the original scale and this is sufficient to get $N_\infty\equiv 1$.
	
	Most of the work is contained in Section \ref{coresec}. A more detailed discussion of the strategy, together with an informal explanation of the technical statements contained in Section \ref{coresec}, is deferred to the beginning of that section.
	
	On the other hand, Section \ref{harmsec} contains two auxiliary facts about harmonic maps, the second of which is rather technical in nature, while Section \ref{multonesec} is more standard and uses a simple covering argument, together with some ideas similar to those in Section \ref{coresec}.
	
	\begin{corollary}\label{strongconv}
		If there is no bubbling or degeneration of the underlying conformal structure, we have strong $W^{1,2}$-convergence $\Phi_k\to\Phi_\infty$. In general we have a bubble tree convergence.
	\end{corollary}
	
	Theorem \ref{main} and Corollary \ref{strongconv} pave the way to obtain meaningful Morse index bounds.
	Indeed, although Theorem \ref{main} does not rule out the possibility of having a surface covered multiple times by $\Phi_\infty$, a crucial advantage of having a parametrization at our disposal is that we have a reasonable definition of Morse index and nullity: they are defined with respect to the area functional and variations in $C^\infty_c(\Sigma_\infty\setminus\set{z_1,\dots,z_s})$, the points $z_1,\dots,z_s$ being the branch points of the immersion $\Phi_\infty$.\footnote{Although we are dealing with a weakly conformal map $\Phi_\infty$, for which area and energy are the same, it is important to remark that the Morse indices for area and energy, denoted $\operatorname{index}_A$ and $\operatorname{index}_E$ respectively, should not be expected to agree. The relationship between the two is a subtle problem: in this direction, we mention the inequality $\operatorname{index}_E(\Psi)\le\operatorname{index}_A(\Psi)\le\operatorname{index}_E(\Psi)+r$ established in \cite{ejiri}, for a branched minimal immersion $\Psi$, where $r=r(g,b)$ depends on the genus and the number of branch points of $\Psi$.}
	
	The natural expected inequalities would be
	\begin{align*}
		\operatorname{index}(\Phi_\infty)\le k\le\operatorname{index}(\Phi_\infty)+\operatorname{nullity}(\Phi_\infty).
	\end{align*}
	
	An abstract framework to show upper bounds for the Morse index, dealing with general penalized functionals on Banach manifolds, is developed in \cite{miche}. Combining Corollary \ref{strongconv} with the general result obtained in \cite{miche} and with \cite{rivlower}, we reach the following conclusion (we refer the reader to \cite{miche} for the notion of \emph{admissible family}).
	
	\begin{corollary}
		Given an admissible family $\mathcal A\subseteq\mathcal{P}(\text{Imm}(\Sigma,\subman))$ of dimension $k$ and calling
		\begin{align*} W_{\mathcal A}:=\inf_{A\in\mathcal{A}}\sup_{\Phi\in A}\operatorname{area}(\Phi) \end{align*}
		the \emph{width} of $\mathcal{A}$, there exists a (possibly branched) minimal immersion $\Phi$ of a closed surface $S$ into $\subman$ such that
		\begin{itemize}[leftmargin=\widthof{(iii) },font=\normalfont]
			\item[(i)] $\operatorname{genus}(S)\le\operatorname{genus}(\Sigma)$,
			\item[(ii)] $W_{\mathcal A}=\operatorname{area}(\Phi)$,
			\item[(iii)] $\operatorname{index}(\Phi)\le k$.
		\end{itemize}
	\end{corollary}
	
	However, proving the second inequality, namely $k\le\operatorname{index}(\Phi)+\operatorname{nullity}(\Phi)$, seems to require a finer understanding of the convergence $\Phi_k\to\Phi_\infty$. We hope to be able to deal with this question elsewhere.
	
	Also, it would be interesting to adapt Gromov's notion of volume spectrum (and higher codimension generalizations), used to produce infinitely many minimal hypersurfaces in many settings, to the present situation. To this aim, a natural topological question concerns how much \emph{genus} is needed to realize a nontrivial $p$-sweepout (in the sense of Gromov--Guth), and how to realize the sweepout within the space of immersions or in an appropriate closure of it.
	
	\section{Notation}\label{notsec}
	
	\begin{itemize}
		\item We will assume, without loss of generality, that $\subman$ is isometrically embedded in some Euclidean space $\R^q$. Given $p\in\subman$ and $\ell>0$, we set $\subman_{p,\ell}:=\ell^{-1}(\subman-p)$.
	
		\item In what follows, $\Pi$ will always denote a 2-plane through the origin, which we identify with the corresponding orthogonal projection $\Pi:\R^q\to\Pi$. We call $\Pi^\perp$ the orthogonal $(q-2)$-subspace, identified with the corresponding orthogonal projection. Given 2-planes $\Pi,\Pi'$, we denote by $\dist(\Pi,\Pi')$ an arbitrary distance on the Grassmannian $\text{Gr}(2,\R^q)$, e.g. the one induced by Pl\"ucker's embedding of $\text{Gr}(2,\R^q)$ into the projectivization of $\Lambda_2\R^q$.
	
		The adjoint maps, which are just the inclusions $\Pi\hookrightarrow\R^q$ and $\Pi^\perp\hookrightarrow\R^q$, are denoted $\Pi^*$ and $(\Pi^\perp)^*$, so that
		\begin{align*} \text{id}_{\R^q}=\Pi^*\Pi+(\Pi^\perp)^*\Pi^\perp. \end{align*}
		Also, $\Pi_0$ is the canonical 2-plane, so that $\Pi_0:\R^q\to\R^2$ is the projection onto the first two coordinates, while $\Pi_0^\perp:\R^q\to\R^{q-2}$ is the projection onto the remaining $q-2$.
	
		\item We call $B_r^2(x)$ the ball of center $x$ and radius $r$ in the plane $\C=\R^2$, while $B_s^q(p)$ will denote the ball of center $p$ and radius $s$ in $\R^q$. Given $p\in\Pi$, we call $B_s^\Pi(p)$ the two-dimensional ball with center $p$ and radius $s$ in $\Pi$, i.e. $B_s^\Pi(p):=B_s^q(p)\cap\Pi$. When the center is not specified, it is always meant to be the origin.
	
		\item Given a function $\Psi\in W^{1,2}(B_r^2(x))$ and $0<s\le r$, the notation $\restr{\Psi}{\de B_s^2(x)}$ always refers to the trace of $\Psi$ on the circle $\de B_s^2(x)$.
	
		\item Given $K\ge 1$, we define the following set of Beltrami coefficients:
		\begin{align*} \mathcal{E}_K:=\set{\mu\in L^\infty(\C,\C):\norm{\mu}_{L^\infty}\le\frac{K-1}{K+1}}. \end{align*}
		We let $\mathcal{D}_K$ denote the set of $K$-quasiconformal homeomorphisms $\varphi:\C\to\C$ such that
		\begin{align*} \varphi(0)=0,\qquad\min_{x\in\de B_1^2}\abs{\varphi(x)}=1. \end{align*}
		If $\varphi\in\mathcal{D}_K$, we have $\varphi\in W^{1,2}_{loc}(\C)$ and $\de_{\bar z}\varphi=\mu\de_{z}\varphi$ for some $\mu\in\mathcal{E}_K$, in the weak sense; we refer the reader to \cite[Chapter~4]{imayoshi} for the basic theory of $K$-quasiconformal homeomorphisms in the plane. Moreover, it is immediate to check that $\varphi$ is a linear map in $\mathcal{D}_K$ if and only if $\varphi(e_1)=e_1'$ and $\varphi(e_2)=\lambda e_2'$, for suitable orthonormal bases $(e_1,e_2)$, $(e_1',e_2')$ inducing the canonical orientation and a suitable $1\le\lambda\le K$.
		
		\item We define
		\begin{align*}
			D(K):=\sup\set{\abs{\varphi(x)};x\in\bar B_1^2,\varphi\in\mathcal{D}_K},\quad s(K):=\inf\set{\abs{\varphi^{-1}(y)};\abs{y}\ge\frac{1}{2},\varphi\in\mathcal{D}_K},
		\end{align*}
		so that $\varphi(\bar B_1^2)\subseteq\bar B_{D(K)}^2$ and $\varphi(\bar B_{s(K)}^2)\subseteq\bar B_{1/2}^2$ for all $\varphi\in\mathcal{D}_K$. The fact that $D(K)<\infty$ and $s(K)>0$ is guaranteed by Corollary \ref{qchomcpt}. We also set
		\begin{align*}
			\eta(K):=\frac{1}{4}\inf\set{\abs{\varphi(x)};x\in\de B_{s(K)^2}^2,\varphi\in\mathcal{D}_K}>0.
		\end{align*}
	
		\item We let $\mathcal{D}_K^\Pi$ denote the set of maps having the form $\Pi^*\circ R\circ\varphi$, where $\varphi\in\mathcal{D}_K$ and $R:\R^2\to\Pi$ is a linear isometry.	
		Given $0<\delta<1$, we call $\mathcal{R}_{K,\delta}^\Pi$ the set of maps in $W^{1,2}(B_1^2,\R^q)$ which are close to some $\psi\in\mathcal{D}_K^\Pi$ on the circles of radii $1,s(K),s(K)^2$, namely we set
		\begin{align*} \mathcal{R}_{K,\delta}^\Pi:=\set{\Psi\in W^{1,2}(B_1^2,\R^q):\min_{\psi\in\mathcal{D}_K^\Pi}\max_{r\in\set{1,s(K),s(K)^2}}\norm{\restr{\Psi}{\de B_r^2}(r\cdot)-\psi(r\cdot)}_{L^\infty(\de B_1^2)}\le\delta}. \end{align*}
	
		\item Given $\Psi\in C^1(\Omega,\R^q)$, a ball $B_r^2(z)\cptsub\Omega$ and a 2-plane $\Pi$, we define the \emph{projected multiplicity} function
		\begin{align*} N(\Psi,B_r^2(z),\Pi):\Pi\to\N\cup\set{\infty},\qquad N(\Psi,B_r^2(z),\Pi)(p):=\#((\Pi\circ\Psi)^{-1}(p)\cap B_{r}^2(z)) \end{align*}
		and, given $p\in\Pi$ and $t>0$, we also define the \emph{macroscopic multiplicity}
		\begin{align}\label{macromult} n(\Psi,B_r^2(z),B_t^\Pi(p)):=\Big\lfloor\media_{B_t^\Pi(p)}N(\Psi,B_r^2(z),\Pi)+\frac{1}{2}\Big\rfloor\in\N. \end{align}
		The mean appearing in \eqref{macromult} is finite by the area formula and $\floor{\cdot}$ denotes the integer part. Note that, if the mean is close to an integer $k$,
		then the macroscopic multiplicity is precisely $k$.
		Note also that for any $p\in\R^\envdim$ we have
		\begin{align*}
			&n(\Psi,B_r^2(z),B_t^\Pi(\Pi(p)))=n\pa{\frac{\Psi(z+r\cdot)-p}{t},B_1^2,B_1^\Pi}.
		\end{align*}
	\end{itemize}
	
	\section{Background on parametrized stationary varifolds}\label{backsec}
	
	Let $\subman\subset\R^\envdim$ be a (smooth, closed) embedded submanifold.
	Assume we have a smooth conformal immersion $\Phi:B_1^2\to\subman$, critical for the functional
	\begin{align}\label{functional} \Phi\mapsto\int_{B_1^2}\,d\operatorname{vol}_{g_\Phi}+\sigma^2\int_{B_1^2}(1+|A_{g_\Phi}|_{g_\Phi}^2)^2\,d\operatorname{vol}_{g_\Phi} \end{align}
	with respect to smooth variations $(\Phi_t)$ such that $\Phi_t$ is an immersion agreeing with $\Phi$ outside some compact set $F\subset B_1^2$ (independent of $t$).
	Here $g_\Phi:=\Phi^*g_{\R^\envdim}$ and $A_{g_\Phi}$ is the second fundamental form of $\Phi$. Assume that the following \emph{entropy condition}
	\begin{align}\label{all2} \sigma^2\log(\sigma^{-1})\int_{B_1^2}(1+|A|^2)^2\,d\operatorname{vol}_{g_\Phi}\le\epsilon\int_{B_1^2}\,d\operatorname{vol}_{g_\Phi} \end{align}
	holds for some $\epsilon>0$. This condition plays a fundamental role in the viscosity approach presented in \cite{rivminmax},
	and can be enforced ultimately owing to Struwe's monotonicity trick (see \cite[Section~II]{rivminmax} and the references therein).
	Note that
	\begin{align*}
		&g_{\Phi}=\mz|\nabla\Phi|^2\delta,\quad\int_{B_1^2}\,d\operatorname{vol}_{g_\Phi}=\mz\int_{B_1^2}\abs{\nabla\Phi}^2
	\end{align*}
	by conformality of $\Phi$.
	
	Given any $0<\ell<1$ and $p\in\subman$, recall that $\subman_{p,\ell}=\ell^{-1}(\subman-p)$. The rescaled map
	\[ \Psi:B_1^2\to\subman_{p,\ell},\qquad\Psi:=\ell^{-1}(\Phi-p) \]
	is critical for the functional
	\begin{align}\label{rescfunc} \int_{B_1^2}\,d\operatorname{vol}_{g_\Psi}+\tau^2\int_{B_1^2}(\ell^2+|A|^2)^2\,d\operatorname{vol}_{g_\Psi},\qquad\tau:=\sigma\ell^{-2} \end{align}
	and, being $\tau^2\log(\tau^{-1})\le \ell^{-4}\sigma^2\log(\sigma^{-1})$, it satisfies
	\begin{align}\label{cond2} \tau^2\log(\tau^{-1})\int_{B_1^2}(\ell^2+|A|^2)^2\,d\operatorname{vol}_{g_\Psi}\le\epsilon\int_{B_1^2}\,d\operatorname{vol}_{g_\Psi}, \end{align}
	where now $A$ denotes the second fundamental form of $\Psi$ in $\subman_{p,\ell}$ and its norm is meant with respect to the induced metric $g_\Psi$.
	
	In the sequel, we will establish many intermediate results on maps $\Psi$ arising in this way, by means of compactness arguments. The starting point in these arguments is that, heuristically, if we have sequences $\Psi_k$, $p_k$, $\ell_k\to 0$, $\tau_k\to 0$ and $\epsilon_k\to 0$, then by \eqref{rescfunc} and \eqref{cond2} $\Psi_k$ should have a subsequential limit $\Psi_\infty$ (in some weak sense) which is critical for the area functional in the tangent space $T_{p_\infty}\subman$ (where $p_\infty$ is a subsequential limit of the sequence $p_k$), i.e. $\Psi_\infty$ should be a minimal parametrization.
	
	The kind of limiting object that we get is specified by the following definition.
	
	\begin{definition}\label{psv}
		A triple $(\Omega,\Phi,N)$, with $\Omega\subseteq\C$ open, $\Phi\in W^{1,2}(\Omega,\R^\envdim)$ weakly conformal and $N\in L^\infty(\Omega,\N\setminus\set{0})$, is a \emph{local parametrized stationary varifold} if for almost every $\omega\cptsub\Omega$ the rectifiable $2$-varifold
		\[ \vfd_\omega:=(\Phi(\good\cap\omega),\theta_\omega),\qquad\theta_\omega(p):=\sum_{x\in\good\cap\omega\cap\Phi^{-1}(p)}N(x) \]
		is stationary in the open set $\R^q\setminus\Phi(\de\omega)$, where $\mathcal{G}$ denotes the set of Lebesgue points for both $\Phi$ and $d\Phi$. We also require the technical condition
		\begin{equation*}
		\norm{\vfd_\Omega}(B_s^\envdim(p))=\mz\int_{\Phi^{-1}(B_s^\envdim(p))}N\abs{\nabla\Phi}^2\,d\mathcal{L}^2=O(s^2),
		\end{equation*}
		uniformly in $p\in\R^\envdim$.
	\end{definition}
	
	We refer the reader to \cite[Definition~2.1]{pigriv} for the notion of \emph{almost every domain}, as well as to \cite[Definitions~2.2~and~2.9]{pigriv} for another possible definition (in the global and local versions, respectively), whose equivalence with Definition \ref{psv} is detailed in \cite[Remark~2.3]{pigriv}. The latter formulation will not be used here.
	
	As already mentioned in the introduction, the main result of \cite{pigriv} is that $\Phi$ is harmonic (namely, it coincides a.e. with a harmonic map) and $N$ is (a.e.) constant on those connected components of $\Omega$ where $\Phi$ is not itself (a.e.) constant.
	
	Since these facts are crucially used in many intermediate steps towards the proof of Theorem \ref{main}, we give a precise statement that summarizes all the information we need to extract from the works \cite{rivminmax} and \cite{pigriv}.
	
	\begin{thm}[limiting behaviour of almost flat critical immersions]\label{blackbox}
		Assume that $\Psi_k\in C^2(\bar B_R^2,\subman_{p_k,\ell_k})$ is a sequence of conformal immersions such that $\Psi_k$ is critical for the functional \eqref{rescfunc} on the interior $B_R^2$ (with $\tau_k,\ell_k$ in place of $\tau,\ell$) and
		\begin{itemize}
		\item $\displaystyle\Psi_k|_{\de B_R^2}\to \gamma_\infty$ uniformly, for some $\gamma_\infty:\de B_R^2\to \R^\envdim$,
		\item $\displaystyle\mz\int_{B_R^2}\abs{\nabla\Psi_k}^2\le E$,
		\item $\displaystyle\tau_k^2\log(\tau_k^{-1})\int_{B_R^2}\abs{A}^4\,d\operatorname{vol}_{g_{\Psi_k}}\to 0$,
		\item $\displaystyle\ell_k,\tau_k\to 0$.
		\end{itemize}
		Then, up to subsequences, $\Psi_k\weakto\Psi_\infty$ in $W^{1,2}(B_R^2,\R^q)$, for some $\Psi_\infty$ which is continuous (in the interior), has trace $\gamma_\infty$, and satisfies the convex hull property, namely
		\begin{align*}
			&\Psi_\infty(\bar\omega)\subseteq\operatorname{co}(\Psi_\infty(\de\omega))\qquad\text{for all }\omega\cptsub B_R^2.
		\end{align*}
		The image measures $(\Psi_k)_*\pa{\mz|\nabla\Psi_k|^2}$ in $\R^\envdim$ form a tight sequence.
		
		Given $\omega\cptsub B_R^2$ with $\Psi_\infty(\bar\omega)\subseteq\R^\envdim\setminus\gamma_\infty(\de B_R^2)$, there exists a quasiconformal homeomorphism $\varphi_\infty$ of $\R^2$ and a locally constant multiplicity $N_\infty\in L^\infty(\omega,\Z^+)$ such that the $2$-varifolds induced by $\restr{\Psi_k}{\omega}$ subsequentially converge on $\R^\envdim\setminus\Phi_\infty(\de\omega)$ to a (local) parametrized stationary varifold
		\begin{align*} (\varphi_\infty(\omega),\Psi_\infty\circ\varphi_\infty^{-1},N_\infty\circ\varphi_\infty^{-1}). \end{align*}
		in the varifold sense, namely in duality with $C_c^0((\R^\envdim\setminus\Phi_\infty(\de\omega))\times\operatorname{Gr}(2,\R^\envdim))$.
		Also, on $\omega$ we have the convergence of Radon measures
		\begin{align}\label{blackboxmeas}
		&\mz\abs{\nabla\Psi_k}^2\mathcal{L}^2\weakstarto N_\infty\abs{\de_1\Psi_\infty\wedge\de_2\Psi_\infty}\mathcal{L}^2.
		\end{align}
		Setting $\Gamma_\infty:=\gamma_\infty(\de B_R^2)$, we have $N_\infty\le\frac{E}{\pi\operatorname{dist}(\Psi_\infty(\bar\omega),\Gamma_\infty)^2}$ a.e. and the distortion constant of $\varphi_\infty$ is bounded by $\pa{\frac{E}{\pi\operatorname{dist}(\Psi_\infty(\bar\omega),\Gamma_\infty)^2}}^2$.
	\end{thm}
	
	\begin{proof}[Proof of Theorem \ref{blackbox}]
		The proof is essentially already contained in \cite{rivminmax} and \cite{pigriv}, so we just present the required adaptations.
		
		Up to subsequences, we can assume that $\Psi_k$ has a weak limit $\Psi_\infty$ in $W^{1,2}(B_R^2,\R^q)$, with trace $\gamma_\infty$, and that the varifolds $\vfd_k$ induced by $\Psi_k$ converge to a varifold $\vfd_\infty$ in $\R^q$.
		
		The arguments used in \cite[Section~III]{rivminmax} and in \cite[Section~2]{pigriv} show that $\Psi_\infty$ has a continuous representative (on the interior $B_R^2$), satisfying the convex hull property.
		Also, from \cite[Section~III]{rivminmax} we have that
		\begin{align}\label{stat} \vfd_\infty\text{ is stationary in }U:=\R^\envdim\setminus\gamma_\infty(\de B_R^2) \end{align}
		and is an integer rectifiable varifold.
		We claim that the measures $\|\vfd_k\|=(\Psi_k)_*\pa{\mz|\nabla\Psi_k|^2}$ on $\R^\envdim$ form a tight sequence.
		If this were not true, up to subsequences we could find points $q_k\in\R^\envdim$ with $|q_k|\to\infty$ and such that
		the argument of \cite[Lemma~III.3]{rivminmax} applies (with $q_k$ in place of $q$), on the region $\R^\envdim\setminus\Psi_k(\de B_R^2)$.
		Hence,
		\begin{align*}
			&\liminf_{k\to\infty}\|\vfd_{k}\|(B_1^\envdim(q_k))>0.
		\end{align*}
		So the varifolds $\vfd_k-q_k$ converge subsequentially to a nontrivial varifold $\vfd_\infty'$ in $\R^q$, stationary on $U':=\R^\envdim$: indeed, the proof of \eqref{stat} can be repeated with $\Psi_k-q_k$ in place of $\Psi_k$ (and $U'$ in place of $U$), using the fact that, for all $s>0$, the image of $\Psi_k|_{\de B_R^2}-q_k$ is eventually disjoint from $B_s^\envdim$. Its total mass $\|\vfd_\infty'\|(\R^\envdim)$ must be bounded by $\liminf_{k\to\infty}\|\vfd_k\|(\R^\envdim)\le E$; however, the monotonicity formula implies that $\|\vfd_\infty'\|(\R^\envdim)=\infty$, a contradiction.
		
		We also claim that the bubbling set is empty in our setting. Indeed, by tightness of the measures $\|\vfd_k\|$,
		a bubbling point would produce, in the limit, a nontrivial stationary varifold in $\R^q$. Again, its mass would be bounded by $E$, contradicting the monotonicity formula.
		
		Now \cite[Lemma~III.5]{rivminmax} gives, up to further subsequences, the limit
		\begin{align*}
			&\nu_k:=\mz|\nabla\Psi_k|^2\weakstarto\nu_\infty,\quad\text{with }\nu_\infty=m\,\mathcal{L}^2
		\end{align*}
		in the sense of Radon measures (i.e. in duality with $C^0_c(B_R^2)$). The function $m(z)\ge 0$ equals $N_\infty(z)|\de_1\Phi_\infty\wedge\de_2\Phi_\infty|(z)$ a.e., for a positive integer $N_\infty(z)$ which is bounded by the density of $\vfd_\infty$ at $\Psi_\infty(z)$ whenever $\Psi_\infty(z)\in U$.
		
		Let $\omega\cptsub B_R^2$ be such that $\Psi_\infty(\bar\omega)\subseteq\R^\envdim\setminus\Gamma_\infty$.
		Defining $s:=\operatorname{dist}(\Psi_\infty(\bar\omega),\Gamma_\infty)>0$, note that
		\begin{align*}
			&B_s^\envdim(q)\subseteq U\qquad\text{for all }q\in\Psi_\infty(\bar\omega).
		\end{align*}
		Hence, by the monotonicity formula, the density of $\vfd_\infty$ at such points $q$ is bounded by
		$\frac{E}{\pi s^2}$. This gives the upper bound for $N_\infty$. As explained in detail in \cite[Section~4]{pigriv},
		there exists a quasiconformal homeomorphism $\varphi_\infty$ of the plane, with distortion constant bounded by the square of the (essential) supremum of $N_\infty|_\omega$, such that $\Psi_\infty\circ\varphi_\infty^{-1}$ is weakly conformal on $\varphi_\infty(\omega)$. Finally, it is the main outcome of \cite{rivminmax} that the $2$-varifolds induced by $\Phi_k|_\omega$ converge to the (local) parametrized stationary varifold $(\varphi_\infty(\omega),\Psi_\infty\circ\varphi_\infty^{-1},N_\infty\circ\varphi_\infty^{-1})$
		(whose mass measure is bounded above by $\|\vfd_\infty\|$), in the complement of $\Psi_\infty(\de\omega)$.\footnote{The convergence actually holds on all of $\R^\envdim$ (or, more precisely, on $\R^\envdim\times\operatorname{Gr}(2,\R^\envdim)$) if $\nu_\infty(\de\omega)=0$.}
	\end{proof}

	\begin{thm}[regularity of parametrized stationary varifolds]\label{regthm}
		In the situation of Theorem \ref{blackbox}, $\Psi_\infty\circ\varphi_\infty^{-1}:\varphi_\infty(\omega)\to\R^\envdim$ is harmonic. Also, if $\omega$ is connected and $\Psi_\infty|_\omega$ is not constant, $N_\infty$ equals a constant integer (a.e.) on $\omega$ and $\Psi_\infty\circ\varphi_\infty^{-1}$ is a minimal branched immersion.
	\end{thm}

	\begin{proof}[Proof of Theorem \ref{regthm}]
		This is a special case of the main theorem in \cite{pigriv}, namely \cite[Theorem~5.7]{pigriv}.
	\end{proof}
	
	\section{Two lemmas on harmonic maps}\label{harmsec}

	\begin{lemmaen}[uniform convergence for Dirichlet problem with variable domain]\label{barrier}
		Let $\gamma_k\in C^0(\de B_1^2,\R^2)$ be a sequence of Jordan curves converging (in $C^0$) to a Jordan curve $\gamma_\infty$ and let $f_k\in C^0(\de B_1^2)$ be a sequence converging uniformly to a function $f_\infty$.
		Let $D_k$ be the domain bounded by $\gamma_k$, let $u_k\in C^0(\bar D_k)$ be the harmonic extension of $f_k\circ\gamma_k^{-1}$, and similarly define $D_\infty$ and $u_\infty$. Then $u_k\to u_\infty$ in $C^0_{loc}(D_\infty)$. Moreover, if $y_k\to y_\infty$ with $y_k\in\bar D_k$ and $y_\infty\in\bar D_\infty$, then $u_k(y_k)\to u_\infty(y_\infty)$.
	\end{lemmaen}

	Note that such harmonic extensions exist and are unique, since by Carath\'eodory's theorem there exist homeomorphisms $\bar B_1^2\to\bar D_k$ restricting to biholomorphisms $B_1^2\to D_k$ (and similarly for $D_\infty$),
	allowing to reduce matters to the well-known existence and uniqueness of the harmonic extension on the unit disk.

	\begin{proof}[Proof of Lemma \ref{barrier}]
		Since the functions $f_k$ are equibounded, from the maximum principle and interior estimates it follows that the functions $u_k$ are equibounded in $C^2(\bar\omega)$, for any $\omega\cptsub D_\infty$, and hence by Ascoli--Arzel\`a theorem the convergence $u_k\to u_\infty$ in $C^0_{loc}(D_\infty)$ follows from the second claim.
		
		It suffices to show that the second claim holds for a subsequence: once this is done, it can be obtained for the full sequence by a standard contradiction argument (given a sequence $y_k\to y_\infty$, if $u_k(y_k)$ does not converge to $u_\infty(y_\infty)$, we can find a subsequence such that it converges to a different value; then we reach a contradiction along a further subsequence where the second claim holds).
		
		Up to removing a finite set of indices, we can suppose that there is a point $p$ such that $p\in D_k$ for all $k\in\N\cup\set{\infty}$. By Carath\'eodory's theorem, we can find homeomorphisms $\upsilon_k:\bar B_1^2\to\bar{D}_k$ restricting to biholomorphisms from $B_1^2$ to $D_k$, so that $\restr{\upsilon_k}{\de B_1^2}=\gamma_k\circ\beta_k$, for suitable homeomorphisms $\beta_k:\de B_1^2\to\de B_1^2$ (for all $k\in\N$), and $\upsilon_k(0)=p$.
		
		Since the maps $\upsilon_k$ and $\upsilon_k^{-1}$ are equibounded and harmonic, we can assume that
		\begin{align} \upsilon_k\to\upsilon_\infty,\qquad\zeta_k:=\upsilon_k^{-1}\to\zeta_\infty \end{align}
		in $C^\infty_{loc}(B_1^2)$ and $C^\infty_{loc}(D_\infty)$, respectively. Note that $\upsilon_\infty$ is a holomorphic map taking values into $\bar D_\infty$, while $\zeta_\infty$ is holomorphic and takes values into $B_1^2$ (by the maximum principle, since $\zeta_\infty(p)=0$ and $\abs{\zeta_\infty}\le 1$). So for any $w\in D_\infty$ the set $\set{\zeta_k(w)\mid k\in\N}\cup\set{\zeta_\infty(w)}\subset B_1^2$ is compact and we infer
		\begin{align} \upsilon_\infty\circ\zeta_\infty(w)=\lim_{k\to\infty}\upsilon_k\circ\zeta_k(w)=w. \end{align}
		Hence $\upsilon_\infty$ is surjective and thus an open map. So $\upsilon_\infty(B_1^2)=D_\infty$ and, by \cite[Theorem~10.43]{rudin} (applied with $f:=\upsilon_\infty-w$, $g:=\upsilon_k-w$, for a fixed $w\in D_\infty$ and an arbitrary circle $\de B_r^2\subseteq B_1^2$ avoiding $\upsilon_\infty^{-1}(w)$, with $k$ large enough), it is also injective.		
		By Carath\'eodory's theorem, it extends continuously to a homeomorphism (still denoted $\upsilon_\infty$) from $\bar B_1^2$ to $\bar D_\infty$ and we have $\restr{\upsilon_\infty}{\de B_1^2}=\gamma_\infty\circ\beta_\infty$ for a suitable homeomorphism $\beta_\infty:\de B_1^2\to\de B_1^2$.
		
		Up to subsequences, applying Helly's selection principle to suitable lifts $\bar\beta_k:\R\to\R$, we can assume that $\beta_k\to\tilde\beta_\infty$ everywhere, for some order-preserving $\tilde\beta_\infty$.\footnote{The map $\tilde\beta_\infty$ could also be order-reversing: this happens precisely if $\beta_k$ reverses the orientation along the subsequence. For simplicity, we assume $\beta_k$, $\tilde\beta_\infty$ to be order-preserving (the other case being analogous).}
		On the other hand, since $\sup_k\int_{B_1^2}\abs{\upsilon_k'}^2=\sup_k\mathcal{L}^2(D_k)$ is finite, we have weak convergence $\upsilon_k\weakto\upsilon_\infty$ in $W^{1,2}(B_1^2)$ and thus weak convergence $\gamma_k\circ\beta_k\weakto\gamma_\infty\circ\beta_\infty$ in $L^2(\de B_1^2)$. The everywhere convergence $\gamma_k\circ\beta_k\to\gamma_\infty\circ\tilde\beta_\infty$ implies $\gamma_\infty\circ\beta_\infty=\gamma_\infty\circ\tilde\beta_\infty$ a.e. and thus $\beta_\infty=\tilde\beta_\infty$ a.e. In particular, $\beta_\infty$ is also order-preserving. Since $\beta_\infty$ is continuous and both maps are order-preserving, we conclude that $\beta_\infty=\tilde\beta_\infty$ everywhere. Using again the continuity of $\beta_\infty$, as well as the everywhere convergence of the order-preserving maps $\beta_k\to\beta_\infty$, we also get that $\beta_k\to\beta_\infty$ uniformly.
		
		Being $\upsilon_k$ the harmonic extension of $\gamma_k\circ\beta_k$ (for $k\in\N\cup\set{\infty}$), we conclude that $\upsilon_k\to\upsilon_\infty$ in $C^0(\bar B_1^2)$. Let $U_k\in C^0(\bar B_1^2)$ be the harmonic extension of $f_k\circ\beta_k$ and note that $U_k\to U_\infty$ in $C^0(\bar B_1^2)$. By conformal invariance, $u_k:=U_k\circ\upsilon_k^{-1}$ is the harmonic extension of $f_k\circ\gamma_k^{-1}$ on $D_k$ (for $k\in\N\cup\set{\infty}$).
		
		Finally, we claim that in the situation of the second claim we have $\upsilon_k^{-1}(y_k)\to\upsilon_\infty^{-1}(y_\infty)$. This easily follows from the injectivity of $\upsilon_\infty$: if we had $\abs{\upsilon_{k}^{-1}(y_{k})-\upsilon_\infty^{-1}(y_\infty)}\ge\epsilon$ along some subsequence (for some $\epsilon>0$), we would have a subsequential limit point $x_\infty\in\bar B_1^2$ with $\abs{x_\infty-\upsilon_\infty^{-1}(y_\infty)}\ge\epsilon$ and $\upsilon_\infty(x_\infty)=\lim_{k\to\infty}y_k=y_\infty$, which is a contradiction. Hence,
		\begin{align} u_k(y_k)=U_k(\upsilon_k^{-1}(y_k))\to U_\infty(\upsilon_\infty^{-1}(y_\infty))=u_\infty(y_\infty), \end{align}
		as desired. 
	\end{proof}
	
	\begin{lemmaen}[injectivity under a boundary constraint]\label{inj}
		Given $K\ge 1$ and $s,\epsilon>0$, there exists a constant $0<\delta_0<\epsilon$, depending only on $q,K,s,\epsilon$, with the following property: whenever
		\begin{itemize}
			\item $\displaystyle\Psi\in W^{1,2}\cap C^0(\bar B_1^2,\R^q)$ has $\norm{\restr{\Psi}{\de B_1^2}-\restr{\psi(s\cdot)}{\de B_1^2}}_{C^0(\de B_1^2)}\le\delta_0$ for some $\psi\in\mathcal{D}_K^\Pi$,
			\item $\displaystyle\Psi\circ\varphi^{-1}$ is harmonic and weakly conformal on $\varphi(B_1^2)$, where $\varphi:\R^2\to\R^2$ is a $K$-quasiconformal homeomorphism,\footnote{The maps $\psi$ and $\varphi$ are not necessarily related to each other.}
		\end{itemize}
		then $\Pi\circ\Psi\circ\varphi^{-1}$ is a diffeomorphism from $\varphi(\bar B_{1/2}^2)$ onto its image and
		\begin{align} \dist(\Pi,\Pi(x))<\epsilon,\qquad\Pi(x):=\text{2-plane spanned by }\nabla(\Psi\circ\varphi^{-1})(x), \end{align}
		for all $x\in\varphi(\bar B_{1/2}^2)$.
		In particular, $\Pi\circ\Psi$ is injective on $\bar B_{1/2}^2$. 
	\end{lemmaen}

	\begin{proof}[Proof of Lemma \ref{inj}]
		Assume by contradiction that, for a sequence $\delta_k\downarrow 0$, there exist maps $\Psi_k:B_1^2\to\R^q$, planes $\Pi_k$, homeomorphisms $\varphi_k:\R^2\to\R^2$ and coefficients $\mu_k$ such that the claim fails with $\delta_0=\delta_k$. By Corollary \ref{qchomcpt}, up to subsequences we have $\Pi_k\to\Pi_\infty$ and $\restr{\Psi_k}{\de B_1^2}\to\Gamma$, where $\Gamma:\de B_1^2\to\R^q$ is the restriction of a map in $\mathcal{D}_K^{\Pi_\infty}$.
		
		We can assume that $\varphi_k\in\mathcal{D}_K$ (replacing $\varphi_k$ with $\frac{\varphi_k-\varphi_k(0)}{\min_{\de B_1^2}|\varphi_k-\varphi_k(0)|}$). By Corollary \ref{qchomcpt}, we can assume that $\varphi_k\to\varphi_\infty$ and $\varphi_k^{-1}\to\varphi_\infty^{-1}$ in $C^0_{loc}(\R^2)$, for some homeomorphism $\varphi_\infty:\R^2\to\R^2$.

		By harmonicity, up to subsequences we get $\Theta_k:=\Psi_k\circ\varphi_k^{-1}\to\Theta_\infty$ in $C^2_{loc}(\varphi_\infty(B_1^2))$, for some $\Theta_\infty:\varphi_\infty(B_1^2)\to \R^q$, 
		so that $\Theta_\infty$ is weakly conformal and harmonic.

		On the other hand, by Lemma \ref{barrier} applied to the sequence of harmonic maps $\Theta_k$ on the Jordan domains $\varphi_k(B_1^2)$, $\Theta_\infty$ is the harmonic extension of $\Gamma\circ\varphi_\infty^{-1}$ and $\Psi_k\to\Theta_\infty\circ\varphi_\infty=:\Psi_\infty$ in $C^0(\bar B_1^2)$. By the maximum principle we have $\Pi_\infty^\perp\circ\Theta_\infty=0$ and thus $\Pi_\infty\circ\Theta_\infty$ is either holomorphic or antiholomorphic on $\varphi_\infty(B_1^2)$ (once $\Pi_\infty$ is identified with $\C$).
		
		Now, given two Jordan domains $U,V\subset\C$, if a holomorphic map $h:U\to\C$ extends to a continuous map $h:\bar U\to\C$ mapping $\de U$ onto $\de V$ homeomorphically, then $h$ maps $U$ diffeomorphically onto $V$.\footnote{Indeed, $\restr{h}{U}$ must be an open map, hence $h(\bar U)\setminus\de V$ is closed and open in $\C\setminus\de V$ and it follows that $h(U)=V$. We can find biholomorphisms $u:B_1^2\to U$ and $v:B_1^2\to V$ extending to homeomorphisms of the closures. The map $g:=v^{-1}\circ h\circ u$ satisfies $g(B_1^2)\subseteq B_1^2$ and maps $\de B_1^2$ to itself homeomorphically. Given $w\in B_1^2$, for $r<1$ close enough to $1$ the loop $g(re^{i\theta})-w$ is homotopic to $g(e^{i\theta})$ in $\C\setminus\set{0}$, so the classical argument principle gives $\#g^{-1}(w)=1$.}
		Being $\restr{\Pi_\infty\circ\Theta_\infty}{\de\varphi_\infty(B_1^2)}=\Pi_\infty\circ\Gamma\circ\varphi_\infty^{-1}$ a Jordan curve, we deduce that $\Pi_\infty\circ\Theta_\infty$ is a diffeomorphism from $\varphi_\infty(B_1^2)$ onto its image.
		
		Fix now a compact neighborhood $F$ of $\varphi_\infty(\bar B_{1/2}^2)$ in $\varphi_\infty(B_1^2)$, with smooth boundary. Since $\Theta_k\to\Theta_\infty$ in $C^1_{loc}(\varphi_\infty(B_1^2))$, we obtain that eventually $\Pi_k\circ\Theta_k$ is a diffeomorphism of $F$ onto its image, with
		\begin{align*} \dist(\Pi_k,\Pi_k(x))<\epsilon,\qquad x\in F. \end{align*}
		The fact that eventually $\varphi_k(\bar B_{1/2}^2)\subseteq F$ yields the desired contradiction.
	\end{proof}
	

	\section{Technical iteration lemmas}\label{coresec}
	
	\subsection{Informal discussion of the results}\label{coresec.informal}
	
	Since the intermediate results contained in this section have rather involved statements, with several different constants and thresholds appearing along the way, we find it helpful to provide an informal explanation of the meaning of these statements and constants, as well as a rough sketch of the underlying ideas in the proofs.
	
	This section contains four important intermediate results, namely Lemmas \ref{closetoint}, \ref{stabdist}, \ref{smallcontrib} and \ref{sqrtsigma}, which all invoke Theorem \ref{blackbox} (except for Lemma \ref{sqrtsigma}) by means of a compactness-and-contradiction argument.
	All statements are about a conformal immersion $\Psi:\bar B_r^2(z)\to\subman_{p,\ell}$, critical for the functional \eqref{rescfunc} (on the interior).
	For simplicity, in this discussion we assume $z=0$ and $r=1$. The first three statements require the following:
	\begin{itemize}
		\item[(i)] a control of the shape of the images of three circles, dictated by a distortion constant $K$; namely we require that
		\begin{align*}
			&\Psi\in\mathcal{R}_{K,\delta_0}^\Pi
		\end{align*}
		for some $2$-plane $\Pi$ and some small $\delta_0$; recall from Section \ref{notsec} that this means that (up to rotations of $\R^\envdim$) $\Psi$ is $C^0$-close to a $K$-quasiconformal homeomorphism $\varphi\in\mathcal{D}_K:\R^2\to\R^2\subseteq\R^\envdim$ on the three circles $\de B_1^2$, $\de B_{s(K)}^2$, $\de B_{s(K)^2}^2$ (it would be far too restrictive to ask for $C^0$-closeness on all of $B_1^2$);
		\item[(ii)] an upper bound $E$ on the Dirichlet energy $\mz\int_{B_1^2}|\nabla\Psi|^2$;
		\item[(iii)] an upper bound $V$ on the area (divided by $\pi$) of the immersed surface $\Psi(B_1^2)\cap B_1^\envdim$, taking into account multiplicity; namely,
		\begin{align*}
			&\int_{\Psi^{-1}(B_1^q)}\,d\operatorname{vol}_{g_\Psi}=\mz\int_{\Psi^{-1}(B_1^q)}|\nabla\Psi|^2\le V\pi,
		\end{align*}
		where $g_{\Psi}$ is the pullback of the Euclidean metric, which equals $\mz|\nabla\Psi|^2\delta$ by conformality;
		this upper bound will give a crucial improvement on the last conclusions of Theorem \ref{blackbox}, as discussed below.
	\end{itemize}

	Also, in the same spirit as Theorem \ref{blackbox}, these lemmas assume $\tau,\ell\ll 1$ and
	\begin{align*}
		&\tau^2\log(\tau^{-1})\int_{B_1^2}|A|^4\,d\operatorname{vol}_{g_\Psi}\ll\int_{B_1^2}\,d\operatorname{vol}_{g_\Psi}.
	\end{align*}
	In Lemmas \ref{stabdist} and \ref{smallcontrib}, the closeness in (i) is measured by a threshold $\delta_0$
	(which will be specified according to Lemma \ref{inj}), while other closeness or smallness constraints will be measured by
	thresholds $\epsilon_0,\epsilon_0',\epsilon_0''$ in Lemmas \ref{closetoint}, \ref{stabdist}, \ref{smallcontrib}, respectively.
	
	Observe that the hypotheses guarantee that $\Pi\circ\Psi$ maps $B_{s(K)}^2$ to a subset of $B_{1/2}^\Pi$ and $\de B_1^2$ to a subset of $\Pi\setminus B_1^\Pi$ (approximately), hence $\Psi(B_{s(K)}^2)$ is far away from $\Psi(\de B_1^2)$. Hence, when arguing by contradiction, we can apply the last part of Theorem \ref{blackbox} and obtain in the limit a parametrized stationary varifold close to $\Psi|_\omega$ (we will choose either $\omega:=B_{s(K)}^2$ or the smaller domain $\omega:=B_{s(K)^2}^2$).	
	The reason to impose the geometric control on three circles, rather than two, is merely technical and is convenient for the proofs.
	
	Lemma \ref{closetoint} says that the \emph{projected multiplicity} $N(\Psi,B_{s(K)^2}^2,\Pi)$ (introduced in Section \ref{notsec}) issued by $\Psi$ from the ball $B_{s(K)^2}^2$ has an average close to a positive integer $k$, on the ball $B_{\eta(K)}^\Pi$.
	It also asserts that this holds for $2$-planes $\Pi'$ close enough to $\Pi$. As a consequence, the corresponding \emph{macroscopic multiplicity} will be precisely $k$.
	
	Observe that the hypotheses guarantee that $\Pi\circ\Psi$ maps $B_{s(K)^2}^2$ approximately to a superset of $B_{\eta(K)}^\Pi$ (see Section \ref{notsec} for the definition of these geometrical constants). Hence, arguing by contradiction, we obtain in the limit a (parametrized) stationary varifold which is close, in the varifold sense, to $\Psi(B_{s(K)^2}^2)$. The constraints on $\Psi$ force this limiting varifold to lie on a $2$-plane, so by the constancy theorem it has constant integer multiplicity on $B_{\eta(K)}^\Pi$, giving a contradiction.
	Note that the volume constraint $V$ is not used here.
	
	As already mentioned in the introduction, we would now like to find a decreasing sequence of radii $r_0:=1,\dots,r_k\approx\sqrt{\tau}$,
	with $r_j$ comparable to $r_{j+1}$, such that the maps $\Psi(r_j\cdot)$ satisfy the same assumptions (with different scales $\ell_0:=\ell,\dots,\ell_k$ in the target).
	The strategy to get Theorem \ref{main} is then to show that the corresponding macroscopic multiplicities $n_j$ \emph{do not change} from one scale to the next one: $n_0=n_1=\cdots=n_k$. At the smallest scale, we will be able to say that the immersed surface $\ell_k^{-1}\Psi(B_{r_k}^2)$ has small second fundamental form in $L^4$, implying a strong graphical control that allows to conclude $n_k=1$ and thus $n_0=1$. In the situation where we will apply this strategy (namely, in Section \ref{multonesec}),
	upon careful selection of the center $z$, it will be easy to impose the ``maximal'' bounds
	\begin{align*}
		&(\ell')^{-2}\int_{\Psi^{-1}(B_{\ell'}^\envdim)}\,d\operatorname{vol}_{g_\Psi}\le V\pi,
		&\tau^2\log(\tau^{-1})\int_{B_{r'}^2}|A|^4\,d\operatorname{vol}_{g_\Psi}\ll\int_{B_{r'}^2}\,d\operatorname{vol}_{g_\Psi},
	\end{align*}
	for \emph{all} $0<\ell'<1$ and $0<r'<1$, by means of covering arguments. However, we \emph{cannot} a priori impose similar bounds on the Dirichlet energy and on the shape of the images of small circles (items (ii) and (i)). Note that if $(\ell')^{-1}\Psi(r'\cdot)$ satisfies (i), then we can bound the Dirichlet energy of this rescaled map on the ball $B_{s(K)}^2$, in terms of $V$, as $\Psi$ maps $B_{s(K)r'}^2$ into $B_{\ell'}^\envdim$ (approximately). So (i) would give (ii) for free (on a smaller domain ball), with a uniform bound (depending on $K$, $V$) in place of $E$.
	
	Lemma \ref{stabdist} is the main technical workhorse, and essentially says that we can circumvent this difficulty: namely, the hypotheses (i)--(iii) are still satisfied for a smaller radius $r'\le\frac{r}{2}$ in the domain, with a smaller scale $\ell'\ell\le\frac{\ell}{2}$ in the codomain. Note that the reference point $p$ also changes; this will in principle destroy the maximal volume bound, but we can still recover (iii) in the new situation, exploiting the fact that the multiplicity is quantized in the limit (see the proof of Lemma \ref{stabdist} and Definition \ref{quanttrick.def} for the details).

	The idea of the proof of Lemma \ref{stabdist} is that, up to a quasiconformal homeomorphism $\varphi$,
	$\Psi$ is close (in the weak $W^{1,2}$-topology) to a conformal harmonic map with small oscillation with respect to $\Pi$.
	Hence, by Lemma \ref{inj}, it will be arbitrarily close to an affine injective conformal map $L$ on smaller and smaller balls $B_{r'}^2$.
	If $\varphi$ were the identity, given a (finite) collection of circles centered at $0$ we would get $C^0$-closeness of $\Psi(r'\cdot)$ to $L$ on all these circles, for some $r'$ small, and we would be done.
	
	The important observation now is that the distortion constant of $\varphi_\infty$ can be bounded solely in terms of $V$: indeed, as in the proof of Theorem \ref{blackbox}, $\Psi(B_1^2)\cap B_1^\envdim$ is close to a stationary varifold $\vfd$ (in $B_1^\envdim$), whose density on $B_{1/2}^\envdim$ is bounded in terms of $V$. Since $\Psi(B_{s(K)}^2)\subseteq B_{1/2}^\envdim$ (approximately), the upper bound on the distortion constant given by Theorem \ref{blackbox} can be improved to a constant $K'(V)$ depending \emph{only} on $V$. Hence, we get (i) also for a smaller radius $r'$ (with $K'(V)$ replacing $K$) and, as already said, this gives also (ii) with a bound $E'(V)$ in place of $E$.
	Our sequence of radii is now obtained by iterated application of Lemma \ref{stabdist} with parameters $K'(V)$, $E'(V)$, $V$.
	
	Given constants $K''$, $E''$ and $V$, which will be chosen when applying these results in Section \ref{multonesec},
	we then fix $K_0:=\max\{K'(V),K''\}$ and $E_0:=\max\{E'(V),E''\}$, so that all the statements apply for all radii $r_0,r_1,\dots,r_k$.
	
	Lemma \ref{smallcontrib} says that the macroscopic multiplicity does not change after applying Lemma \ref{stabdist},
	namely when replacing the domain and codomain scales $r$, $\ell$ with $r'$, $\ell\ell'$ (and $p,\Pi$ with $p',\Pi'$).
	Its proof uses Lemma \ref{inj} to claim that $\Psi$ is approximately a graph over $\Pi$, and then applies the constancy theorem (in the limiting situation).
	
	Finally, as it will be clear along the proof of Theorem \ref{iteration}, Lemma \ref{sqrtsigma} concerns the behaviour of a conformal immersion $\Phi:B_1^2\to\subman$ at a scale (comparable to) $\ell:=\sqrt{\sigma}$ in the codomain, when $\Phi$ is critical for \eqref{functional}.
	Assume that $\Phi(B_r^2)$ has diameter approximately $\ell^2$, and assume the smallness
	\begin{align}\label{ffsmallness}
		&\sigma^2\int_{B_r^2}|A|^4\,d\operatorname{vol}_{g_\Phi}\ll\int_{B_r^2}\,d\operatorname{vol}_{g_\Phi}
	\end{align}
	and the bound $\int_{B_r^2}\,d\operatorname{vol}_{g_\Phi}\le C\ell^2$. When dilating the codomain by a factor $\ell^{-1}$,
	\eqref{ffsmallness} becomes $\int_{B_r^2}|A_\Psi|^4\,d\operatorname{vol}_{g_\Psi}\ll 1$, for $\Psi:=\ell^{-1}(\Phi-\Phi(0))$.
	As $\Psi$ is conformal, we have $\Delta\Psi=2H_\Psi e^{2\lambda}$ (where $e^\lambda$ is the conformal factor).
	Thus, we get that $\Delta\Psi$ is small in $L^4$ provided we can obtain an upper bound on $\lambda$;
	once this is done, by Sobolev's embedding we obtain a $C^1$-control on $\Psi$, which implies that the macroscopic multiplicity is $1$ at this scale.
	
	In order to bound $\lambda$, we use a result by H\'elein (belonging to a broad class of phenomena of integrability by compensation,
	whose study dates back to the discovery of Wente's inequality), guaranteeing the existence of an orthonormal frame
	$\set{e_1(z),e_2(z)}$ for the tangent space of the immersed surface $\Psi$, with a bound on $\norm{\nabla e_i}_{L^2}$ depending only the $L^2$-norm of the second fundamental form. Then we show that
	\begin{align*}
		&-\Delta\lambda=\de_1e_1\cdot\de_2e_2-\de_2e_1\cdot\de_1e_2
	\end{align*}
	and we compare $\lambda$ with the solution $\mu$ to the same equation, with zero boundary conditions on a ball.
	A pointwise bound for $\mu$ now follows from Wente's inequality, from which one easily deduces the desired upper bound for $\lambda$.
	Although not necessary, we will also show how to obtain a pointwise lower bound on $\lambda$ in this situation.	
	
	While reading Sections \ref{coresec} and \ref{multonesec}, it can be useful to refer to the following diagrams, illustrating how the constants depend on each other:
	
	
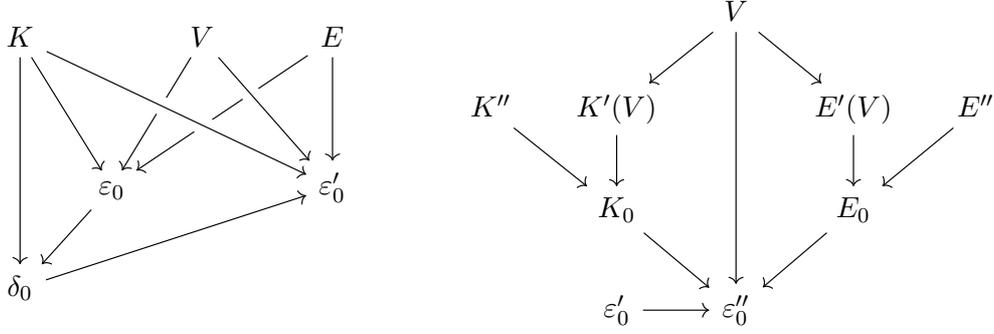
\begin{center}
		\begin{tikzcd}[column sep=15pt]
		K & & V & & E \\
		& & & & \\
		& \epsilon_0 \ar[from=1-1] \ar[from=1-3] \ar[from=1-5] & & & \epsilon_0' \ar[from=1-1,crossing over] \ar[from=1-3,crossing over] \ar[from=1-5] \ar[from=4-1] \\
		\delta_0 \ar[from=1-1] \ar[from=3-2] & & & &
		\end{tikzcd}
		\hspace{30pt}
		\begin{tikzcd}[column sep=15pt]
		& & V & & \\
		K'' & K'(V) \ar[from=1-3] & & E'(V) \ar[from=1-3] & E'' \\
		& K_0 \ar[from=2-1] \ar[from=2-2] & & E_0 \ar[from=2-4] \ar[from=2-5] & \\
		& \epsilon_0' & \epsilon_0'' \ar[from=1-3] \ar[from=3-2] \ar[from=3-4] \ar[from=4-2] & &
		\end{tikzcd}
	\end{center}
	
	\noindent(an arrow $A\to B$ means that $B$ depends on $A$).
	
	\subsection{Rigorous statements and proofs}\label{coresec.proofs}
	
	We now make the above discussion rigorous. In a first reading, it can be helpful to pretend that all quasiconformal homeomorphisms appearing in the proofs coincide with the identity.
	
	\begin{definition}\label{szero}
		Given $V>0$ with $V=\floor{V}+\mz$, we define the constants
		\begin{align*}
			&K'(V):=(4V)^2,\qquad E'(V):=2\pi K'(V) D(K'(V))^2.
		\end{align*}
		In the sequel, it will be convenient to assume always that $V\in\N+\frac{1}{2}$, so that $V=\floor{V}+\mz$.
	\end{definition}

	\begin{lemmaen}[almost integrality of averaged projected multiplicity]\label{closetoint}
		There exists $0<\epsilon_0<\eta(K)$, depending on $E,V>0$, $K\ge 1$ and 
		$\subman$, such that whenever
		$\Psi\in C^2(\bar B_{r}^2(z),\subman_{p,\ell})$ is a conformal immersion, critical for the functional \eqref{rescfunc} on $B_{r}^2(z)$, and $\Pi,\Pi'$ are 2-planes satisfying
		\begin{itemize}
			\item $\displaystyle\Psi(z+r\cdot)\in\mathcal{R}_{K,\epsilon_0}^\Pi$,
			\item $\displaystyle\mz\int_{B_{r}^2(z)}\abs{\nabla\Psi}^2\le E$,
			\item $\displaystyle\int_{\Psi^{-1}(B_1^q)}\,d\operatorname{vol}_{g_\Psi}=\mz\int_{\Psi^{-1}(B_1^q)}\abs{\nabla\Psi}^2\le V\pi$,
			\item $\displaystyle\tau^2\log(\tau^{-1})\int_{B_{r}^2(z)}\abs{A}^4\,d\operatorname{vol}_{g_\Psi}\le\epsilon_0$ for some $0<\tau\le\epsilon_0$,
			\item $\displaystyle\dist(\Pi,\Pi')\le\epsilon_0$ and $0<\ell\le\epsilon_0$,
		\end{itemize}
		then the projected multiplicity $N(\Psi,B_{s(K)^2r}^2(z),\Pi)$ satisfies
		\begin{align} \dist\Big(\media_{B_{\eta(K)}^\Pi}N(\Psi,B_{s(K)^2r}^2(z),\Pi),\,\Z^+\Big)<\frac{1}{8}, \end{align}
		\begin{align}\Bigg|\media_{B_{\eta(K)}^\Pi}N(\Psi,B_{s(K)^2r}^2(z),\Pi)
		-\media_{B_{\eta(K)}^{\Pi'}}N(\Psi,B_{s(K)^2r}^2(z),\Pi')\Bigg|<\frac{1}{8}, \end{align}
		where $\Z^+$ is the set of positive integers.
	\end{lemmaen}
	
	\begin{proof}[Proof of Lemma \ref{closetoint}]
		We can assume $z=0$ and $r=1$. Suppose by contradiction that there exist sequences $\epsilon_k\downarrow 0$, $\tau_k$, $\ell_k$, points $p_k$, maps $\Psi_k$ and planes $\Pi_k,\Pi_k'$ making the claim false for $\epsilon_0=\epsilon_k$. Up to subsequences, we can assume that $\Pi_k,\Pi_k'\to\Pi_\infty$, that $\Psi_k$ has a weak limit $\Psi_\infty$ in $W^{1,2}(B_1^2,\R^q)$, with traces $\restr{\Psi_\infty}{\de B_s^2}(s\cdot)=\psi(s\cdot)$ for some $\psi\in\mathcal{D}_K^{\Pi_\infty}$ and all $s\in\set{1,s(K),s(K)^2}$ (thanks to Corollary \ref{qchomcpt}), and that the varifolds $\vfd_k$ induced by $\Psi_k$ converge to a varifold $\vfd_\infty$ in $\R^q$.
		
		We now invoke Theorem \ref{blackbox}. Recalling the definition of $\mathcal{D}_K^{\Pi_\infty}$ and $s(K)$ from Section \ref{notsec}, the convex hull property satisfied by $\Psi_\infty$ gives
		\begin{align}\label{disttarget1}
		&\Psi_\infty(\bar B_{s(K)}^2)\subseteq\co{\Psi_\infty(\de B_{s(K)}^2)}\subseteq\bar B_{1/2}^\envdim,
		\end{align}
		so that, being $\Gamma_\infty=\Psi_\infty(\de B_1^2)$ disjoint from $B_1^\envdim$,
		\begin{align}\label{disttarget2}
		&\dist(\Psi_\infty(x),\Gamma_\infty)\ge\mz\quad\text{for }x\in\bar B_{s(K)}^2.
		\end{align}
		Theorem \ref{blackbox} gives the varifold convergence $\vfd_k'\weakstarto\vfd_\infty'$ and $\vfd_k''\weakstarto\vfd_\infty''$ as $k\to\infty$, as well as the tightness of the sequences of mass measures $\norm{\vfd_k'}$ and $\norm{\vfd_k''}$, where $\vfd_k'$ and $\vfd_k''$ are the varifolds issued by $\restr{\Psi_k}{B_{s(K)}^2}$ and $\restr{\Psi_k}{B_{s(K)^2}^2}$ respectively, while $\vfd_\infty'$ and $\vfd_\infty''$ are the ones issued by $(\varphi_\infty(B_{s(K)}^2),\Psi_\infty\circ\varphi_\infty^{-1},N_\infty\circ\varphi_\infty^{-1})$ and $(\varphi_\infty(B_{s(K)^2}^2),\Psi_\infty\circ\varphi_\infty^{-1},N_\infty\circ\varphi_\infty^{-1})$.\footnote{The fact that one can choose the same multiplicity $N_\infty$ and the same quasiconformal homeomorphism $\varphi_\infty:\R^2\to\R^2$ for both domains is evident from the proof of Theorem \ref{blackbox}.}
		
		Although not needed in the present proof, let us remark the following improvement on the last statement in Theorem \ref{blackbox}: we have
		\begin{align*}
			&N_\infty\le\frac{V\pi}{\pi\Big(\mz\Big)^2}=4V
		\end{align*}
		and the distortion constant of $\varphi_\infty$ is bounded by $K'(V)=(4V)^2$. Indeed, since
		$\vfd_\infty$ is stationary in $B_1^\envdim$ and $\|\vfd_\infty\|(B_1^\envdim)\le V\pi$, by the monotonicity formula its density is bounded by $\frac{V\pi}{\pi(1-|p|)^2}$ at any $p\in B_1^\envdim$.
		In particular, \eqref{disttarget1} gives an upper bound $4V$ at points in $\Psi_\infty(B_{s(K)}^2)$, which implies our claim.

		The support of $\vfd_\infty''$ is contained in the plane $\Pi_\infty$, by the convex hull property enjoyed by $\Psi_\infty$ and the fact that $\Psi_\infty$ maps $\de B_{s(K)^2}^2$ to $\Pi_\infty$.
		Since $\Psi_\infty(\de B_{s(K)^2}^2)$ does not intersect $B_{\eta(K)}^{\Pi_\infty}$, the varifold $\vfd_\infty''$ is stationary here and thus, by the constancy theorem \cite[Theorem~41.1]{simon}, it has a constant density $\nu\in\N$.
		We must have $\nu>0$, since $\Psi_\infty(B_{s(K)^2}^2)$ is a superset of $B_{\eta(K)}^{\Pi_\infty}$ by Lemma \ref{winded} (applied to $\eta(K)^{-1}\Psi_\infty(s(K)^2\cdot)$).
		The area formula and the tightness of $\norm{\vfd_k''}$ then give
		\begin{align*}
			\media_{B_{\eta(K)}^{\Pi_k}}N(\Psi_k,B_{s(K)^2}^2,\Pi_k)=\frac{\norm{(\Pi_k)_*\vfd_k''}(B_{\eta(K)}^{\Pi_k})}{\pi \eta(K)^2}\to\frac{\norm{(\Pi_\infty)_*\vfd_\infty''}(B_{\eta(K)}^{\Pi_\infty})}{\pi \eta(K)^2}=\nu.
		\end{align*}
		Similarly, $\media_{B_{\eta(K)}^{\Pi_k'}}N(\Psi_k,B_{s(K)^2}^2,\Pi_k')\to\nu$ as $k\to\infty$. Hence the claim is eventually true, yielding the desired contradiction.
	\end{proof}

	We now specify $\delta_0$ so that Lemma \ref{inj} applies, with $\epsilon:=\epsilon_0$ and $s:=s(K)$. Note that $\delta_0<\epsilon_0<\eta(K)$ and that $\epsilon_0$ and $\delta_0$ still depend on $V$, $K$ and $E$.

	\begin{lemmaen}[existence of a smaller good scale]\label{stabdist}
		Given $E>0$ and $K\ge 1$ there exists a constant $0<\epsilon_0'<\epsilon_0$ 
		(depending on $E,V,K,\subman$) with the following property:
		if a conformal immersion $\Psi\in C^2(\bar B_{r}^2(z),\subman_{p,\ell})$ is critical for the functional \eqref{rescfunc} (on the interior) and satisfies
		\begin{itemize}
			\item $\displaystyle\Psi(z+r\cdot)\in\mathcal{R}_{K,\delta_0}^{\Pi}$,
			\item $\displaystyle\mz\int_{B_{r}^2(z)}\abs{\nabla\Psi}^2\le E$,
			\item $\displaystyle\frac{1}{\pi}\int_{\Psi^{-1}(B_1^q)}\,d\operatorname{vol}_{g_\Psi},\,\frac{1}{\pi\eta(K)^2}\int_{\Psi^{-1}(B_{\eta(K)}^q)}\,d\operatorname{vol}_{g_\Psi}\le V$,
			\item $\displaystyle\tau^2\log(\tau^{-1})\int_{B_{r}^2(z)}\abs{A}^4\,d\operatorname{vol}_{g_\Psi}\le\epsilon_0'$ for some $0<\tau\le\epsilon_0'$,
			\item $\displaystyle 0<\ell\le\epsilon_0'$,
		\end{itemize}
		then there exist a new point $p'\in\subman_{p,\ell}$, new scales $r',\ell'$ and a new 2-plane $\Pi'$ with
		\begin{itemize}
			\item $\displaystyle\epsilon_0'r<r'<s(K)r$,
			\item $\displaystyle\epsilon_0'<\ell'<\mz$,
			\item $\displaystyle\dist(\Pi,\Pi')<\epsilon_0$,
			\item $\displaystyle(\ell')^{-1}(\Psi(z+r'\cdot)-p')\in\mathcal{R}_{K'(V),\delta_0}^{\Pi'}$,
			\item $\displaystyle\mz\int_{B_{r'}^2(z)}\abs{\nabla\Psi'}^2<E'(V)$, for $\Psi':=(\ell')^{-1}(\Psi-p')$ (defined on $B_{r'}^2(z)$),
			\item $\displaystyle\frac{1}{\pi}\int_{(\Psi')^{-1}(B_1^q)}\,d\operatorname{vol}_{g_{\Psi'}},\,\frac{1}{\pi \eta(K)^2}\int_{(\Psi')^{-1}(B_{\eta(K)}^q)}\,d\operatorname{vol}_{g_{\Psi'}}<\Big\lfloor\pa{\frac{\eta(K)}{\eta(K)-\epsilon_0}}^2 V\Big\rfloor+\mz$.
		\end{itemize}		
	\end{lemmaen}
	
	\begin{proof}[Proof of Lemma \ref{stabdist}]
		We can assume $z=0$ and $r=1$.
		By contradiction, suppose that there is a sequence $\epsilon_k\downarrow 0$ such that the claim fails (with $\epsilon_0'=\epsilon_k$)
		for all radii $\epsilon_k<r'<s(K)$, for some $\Psi_k$ and $\Pi_k$ satisfying all the hypotheses. As observed in the proof of Lemma \ref{closetoint}, up to subsequences we get a limiting local parametrized stationary varifold $(\Omega_\infty,\Theta_\infty,N_\infty\circ\varphi_\infty^{-1})$ in $\R^q$, where $\Theta_\infty=\Psi_\infty\circ\varphi_\infty^{-1}$ and $\Omega_\infty=\varphi_\infty(B_{s(K)}^2)$ for a suitable $K'(V)$-quasiconformal homeomorphism $\varphi_\infty$ of the plane. Moreover, assuming also that $\Pi_k\to\Pi_\infty$ and $p_k\to p_\infty$, by weak convergence of traces and Corollary \ref{qchomcpt} we still have $\Psi_\infty\in\mathcal{R}_{K,\delta_0}^{\Pi_\infty}$.
		By Theorem \ref{regthm}, $\Theta_\infty$ is harmonic.
		Also, it takes values in the tangent space $T$ at $p_\infty$ (translated to the origin).

		We can assume that $\varphi_\infty(0)=0$.
		By definition of $\delta_0$ and Lemma \ref{inj}, applied to $\Psi_\infty(s(K)\cdot)$ and $\varphi_\infty(s(K)\cdot)$, $\Theta_\infty$ is a diffeomorphism from $\varphi_\infty(\bar B_{s(K)/2}^2)$ onto its image and the differential $\nabla\Theta_\infty(0)$ is a conformal linear map of full rank, spanning a plane $\Pi'$ with $\dist(\Pi_\infty,\Pi')<\epsilon_0$.
		
		Since $s(K)^2\le\frac{s(K)}{2}$, the varifolds $\vfd_k$ induced by $\restr{\Psi_k}{B_{s(K)^2}^2}$ converge to $\vfd_\infty$, induced by $(\varphi_\infty(B_{s(K)^2}^2),\Theta_\infty,N_\infty\circ\varphi_\infty^{-1})$.
		Using Lemma \ref{winded}, applied to $\eta(K)^{-1}\Pi_\infty\circ\Psi_\infty(s(K)^2\cdot)$, and the fact that $\delta_0<\eta(K)$, we deduce the existence of a point $y\in B_{s(K)^2}^2$ such that $\Pi_\infty\circ\Psi_\infty(y)=0$. By the convex hull property enjoyed by $\Psi_\infty$, it follows that
		\begin{align*}
			&\abs{\Psi_\infty(y)}=\abs{\Pi_\infty^\perp\circ\Psi_\infty(y)}\le\delta_0,
		\end{align*}
		as $\Psi_\infty(\de B_{s(K)^2}^2)\subseteq\set{p:\abs{\Pi_\infty^\perp(p)}\le\delta_0}$.
		Since $\norm{\vfd_\infty}(B_{\eta(K)}^q)\le V\pi\eta(K)^2$, the stationarity of $\vfd_\infty$ on $B_{\eta(K)}^\envdim$ implies that its density at $\Psi_\infty(y)$ is at most
		\begin{align}\label{smalldens}
			&\frac{V\pi\eta(K)^2}{\pi(\eta(K)-\delta_0)^2}\le\pa{\frac{\eta(K)}{\eta(K)-\epsilon_0}}^2 V.
		\end{align}
		Being $\vfd_\infty$ stationary in the embedded surface $\Theta_\infty(\varphi_\infty(B_{s(K)^2}^2))$, the constancy theorem gives that its density $\theta$ is a constant integer here. Thus we have
		\begin{align}\label{quantwhy}
			\norm{\vfd_\infty}(\bar B_t^q(p_\infty'))<\Bigg(\Big\lfloor\pa{\frac{\eta(K)}{\eta(K)-\epsilon_0}}^2 V\Big\rfloor+\mz\Bigg) \pi t^2,\qquad p_\infty':=\Theta_\infty(0)\in T,
		\end{align}
		for all $t>0$ small enough. Fix now any $r'<s(K)$ such that we have the strong convergence $\Psi_k(r'\cdot)\to\Psi_\infty(r'\cdot)$ in $C^0(\de B_1^2\cup\de B_{s(K)}^2\cup\de B_{s(K)^2}^2)$ along a subsequence.\footnote{This can be obtained by applying e.g. \cite[Lemma~A.5]{pigriv} to the weakly converging $\R^{3\envdim}$-valued maps
		\begin{align*}
			&(\Psi_k,\Psi_k(s(K)\cdot),\Psi_k(s(K)^2\cdot))\weakto(\Psi_\infty,\Psi_\infty(s(K)\cdot),\Psi_\infty(s(K)^2\cdot)).
		\end{align*}} Note that $\lambda^{-1}\varphi_\infty(r'\cdot)\in\mathcal{D}_{K'(V)}$, where $\lambda:=\min_{\abs{x}=r'}\abs{\varphi_\infty(x)}$. Also, the fact that $\Psi_\infty=\Theta_\infty\circ\varphi_\infty$ and the smoothness of $\Theta_\infty$ give
		\begin{align}\label{firstorder}
			\abs{\Psi_\infty(r'x)-\Psi_\infty(0)-\ang{\nabla\Theta_\infty(0),\varphi_\infty(r'x)}}<\frac{\delta_0\abs{\nabla\Theta_\infty(0)}}{\sqrt{2}D(K'(V))}\abs{\varphi_\infty(r'x)}\le\delta_0\ell'
		\end{align}
		if $r'$ is chosen small enough, where $\ell':=\frac{\abs{\nabla\Theta_\infty(0)}}{\sqrt 2}\lambda$ and $x\in\bar B_1^2$. This implies
		\begin{align*}
			&(\ell')^{-1}(\Psi_k(r'\cdot)-p_\infty')\in\mathcal{R}_{K'(V),\delta_0}^{\Pi'}
		\end{align*}
		by conformality of $\nabla\Theta_\infty(0)$. Shrinking $r'$, we can also ensure that $\ell'<\mz$, as well as
		\begin{align}\label{limenergy}\begin{aligned}
			\int_{B_{r'}^2}N_\infty|\de_1\Psi_\infty\wedge\de_2\Psi_\infty|
			&\le\frac{K'(V)}{2}\int_{B_{D(K'(V))\lambda}^2}\abs{\nabla\Theta_\infty}^2 \\
			&<K'(V)(D(K'(V))\lambda)^2\pi\abs{\nabla\Theta_\infty(0)}^2.
		\end{aligned}\end{align}
		Calling $\vfd_\infty'$ the varifold induced by $(\varphi_\infty(B_{r'}^2),(\ell')^{-1}(\Theta_\infty-p_\infty'),N_\infty\circ\varphi_\infty^{-1})$,
		in view of \eqref{quantwhy} we can even guarantee that
		\begin{align*}
			\frac{\norm{\vfd_\infty'}(\bar B_1^q)}{\pi},\,\frac{\norm{\vfd_\infty'}(\bar B_{\eta(K)}^q)}{\pi\eta(K)^2}<\Big\lfloor\pa{\frac{\eta(K)}{\eta(K)-\epsilon_0}}^2 V\Big\rfloor+\mz.
		\end{align*}
		Calling $p_k'$ the closest point to $p_\infty'$ in $\subman_{p_k,\ell_k}$ (eventually defined and converging to $p_\infty'$, since $\subman_{p_k,\ell_k}\to T$), thanks to \eqref{firstorder} and $\lambda^{-1}\varphi_\infty(r'\cdot)\in\mathcal{D}_{K'(V)}$, eventually we have
		\begin{align*}
			&(\ell')^{-1}(\Psi_k(r'\cdot)-p_k')\in\mathcal{R}_{K'(V),\delta_0}^{\Pi'}.
		\end{align*}
		Moreover, \eqref{limenergy} and \eqref{blackboxmeas} give
		\begin{align*}
			\mz\int_{B_{r'}^2(z)}\abs{\nabla\Psi_k}^2\to\int_{B_{r'}^2(z)}N_\infty\abs{\de_1\Psi_\infty\wedge\de_2\Psi_\infty}<(\ell')^2 E'(V).
		\end{align*}
		From the convergence of the varifolds $\vfd_k'$ induced by $\restr{(\ell')^{-1}(\Psi_k-p_k')}{B_{r'}^2}$ to $\vfd_\infty'$ we get
		\begin{align*}
			\limsup_{k\to\infty}\frac{\norm{\vfd_k'}(B_1^q)}{\pi},\,\limsup_{k\to\infty}\frac{\norm{\vfd_k'}(B_{\eta(K)}^q)}{\pi\eta(K)^2}<\Big\lfloor\pa{\frac{\eta(K)}{\eta(K)-\epsilon_0}}^2 V\Big\rfloor+\mz.
		\end{align*}
		So eventually $(\ell')^{-1}(\Psi_k(r'\cdot)-p_k')$ satisfies all the conclusions. This yields the desired contradiction.
	\end{proof}

	\begin{definition}\label{quanttrick.def}
		Given constants $K''\ge 1$ and $E''>0$, we define $K_0:=\max\set{K'(V),K''}$ and $E_0:=\max\set{E'(V),E''}$. We also let $s_0:=s(K_0)$ and $\eta_0:=\eta(K_0)$.
		
		We fix $\epsilon_0$ (and thus $\delta_0$) and $\epsilon_0'$ so that Lemmas
		\ref{closetoint} and \ref{stabdist} apply with $K:=K_0$, $E:=E_0$. Since $\epsilon_0$ depends on $V$, we can assume that it is chosen so small that
		\begin{align}\label{quanttrick} \Big\lfloor\pa{\frac{\eta_0}{\eta_0-\epsilon_0}}^2 V\Big\rfloor+\mz=\floor{V}+\mz=V. \end{align}
		This makes the last conclusion of Lemma \ref{stabdist} match one of the hypotheses, making it possible to iterate that result. On the other hand, the constants $V$, $K''$, $E''$ (upon which all the aforementioned constants depend) will be fixed only in Section \ref{multonesec}.
	\end{definition}
	
	\begin{lemmaen}[$\bm{n(\cdot)}$ does not change from one scale to the next one]\label{smallcontrib}
		There exists a constant $0<\epsilon_0''\le\frac{(\epsilon_0')^2}{2}<\epsilon_0'$ with the following property:
		if a conformal immersion $\Psi\in C^2(\bar B_{r}^2(z),\subman_{p,\ell})$ satisfies the hypotheses of the previous lemma (with $\epsilon_0'',E_0,K_0$ in place of $\epsilon_0',E,K$), then the new point $p'$ and the new radius $r'$ provided by Lemma \ref{stabdist} satisfy
		\begin{align}\label{eq:sameint}
		n(\Psi,B_{s_0^2r}^2(z),B_{\eta_0}^\Pi)
		=n(\Psi-p',B_{s_0^2r'}^2(z),B_{\eta_0\ell'}^\Pi)
		=n(\Psi-p',B_{s_0^2r'}^2(z),B_{\eta_0\ell'}^{\Pi'}).
		\end{align}
	\end{lemmaen}
	
	\begin{proof}[Proof of Lemma \ref{smallcontrib}]
		The second equality in \eqref{eq:sameint} follows immediately from Lemma \ref{closetoint} (applied with $(\ell')^{-1}(\Psi-p')$ on $B_{r'}^2$), which gives
		\begin{align*}
		&n(\Psi-p',B_{s_0^2r'}^2,B_{\eta_0\ell'}^\Pi)=n(\Psi-p',B_{s_0^2r'}^2,B_{\eta_0\ell'}^{\Pi'})
		\end{align*}
		since $\dist(\Pi',\Pi)<\epsilon_0$.
		
		Assume again $z=0$, $r=1$ and, by contradiction, that the first equality in \eqref{eq:sameint} fails, so that we have again two sequences $\epsilon_k\downarrow 0$ and $\Psi_k$. We can assume that $\Pi_k\to\Pi_\infty$, $p_k'\to p_\infty'$, $\ell_k'\to\ell_\infty'$ and $r_k'\to r_\infty'$, with $p_\infty'\in\subman$, $\epsilon_0'\le\ell_\infty'\le\frac{1}{2}$ and $\epsilon_0'\le r_\infty'\le s_0$. Moreover, as in the proof of Lemma \ref{stabdist}, up to further subsequences we get a limiting local parametrized stationary varifold  $(\Omega_\infty,\Theta_\infty,N_\infty\circ\varphi_\infty^{-1})$ in $\R^q$, with $\Omega_\infty=\varphi_\infty(B_{s_0}^2)$. From Theorem \ref{regthm} we know that $\Theta_\infty$ is harmonic and $N_\infty=\nu$ is constant, so Lemma \ref{inj} gives that $\Pi_\infty\circ\Theta_\infty$ is a diffeomorphism from $\varphi_\infty(\bar B_{s_0/2}^2)$ onto its image.
		
		Calling $\vfd_k$ the varifold issued by $\restr{\Psi_k}{B_{s_0^2}^2}$ and $\vfd_\infty$ the one issued by $(\varphi_\infty(B_{s_0^2}^2),\Theta_\infty,\nu)$, we have the varifold convergence $\vfd_k\weakstarto\vfd_\infty$ as $k\to\infty$. The area formula gives
		\begin{align*}
			\media_{B_{\eta_0}^{\Pi_k}}N(\Psi_k,B_{s_0^2}^2,\Pi_k)=\frac{\norm{(\Pi_k)_*\vfd_k}(B_{\eta_0}^{\Pi_k})}{\pi\eta_0^2}\to\frac{\norm{(\Pi_\infty)_*\vfd_\infty}(B_{\eta_0}^{\Pi_\infty})}{\pi\eta_0^2}=\nu,
		\end{align*}
		since $(\Pi_\infty)_*\vfd_\infty$ equals an open superset of $B_{\eta_0}^{\Pi_\infty}$ in $\Pi_\infty$ (by Lemma \ref{winded}), equipped with the constant integer multiplicity $\nu$. Hence, $n(\Psi_k,B_{s_0^2}^2,B_{\eta_0}^{\Pi_k})=\nu$ eventually.
		
		Similarly, calling $\vfd_k'$ the varifold induced by $\restr{\Psi_k}{B_{s_0^2r_k'}^2}$ and $\vfd_\infty'$ the varifold induced by $(\varphi_\infty(B_{s_0^2r_\infty'}^2),\Theta_\infty,\nu)$, we have $\vfd_k'\weakstarto\vfd_\infty'$ as $k\to\infty$, as is readily seen by approximating with domains which do not vary along the sequence. Since $(\ell_\infty')^{-1}(\Psi_\infty(r_\infty'\cdot)-p_\infty')\in\mathcal{R}_{K_0,\delta_0}^{\Pi_\infty}$, again $(\Pi_\infty)_*\vfd_\infty'$ equals a superset of $B_{\eta_0\ell_\infty'}^{\Pi_\infty}$ in $\Pi_\infty$, with constant density $\nu$. This gives again
		\begin{align*}
			\media_{B_{\eta_0\ell_k'}^{\Pi_k}(q_k)}N(\Psi_k,B_{s_0^2r_k'}^2,\Pi_k)=\frac{\norm{(\Pi_k)_*\vfd_k'}(B_{\eta_0\ell_k'}^{\Pi_k}(q_k))}{\pi\eta_0^2(\ell_k')^2}\to\frac{\norm{(\Pi_\infty)_*\vfd_\infty'}(B_{\eta_0\ell_\infty'}^{\Pi_\infty}(q_\infty))}{\pi\eta_0^2(\ell_\infty')^2}=\nu,
		\end{align*}
		where $q_k:=\Pi_k(p_k')$ for $k\in\N\cup\set{\infty}$. Hence, $n(\Psi_k-p_k',B_{s_0^2r_k'}^2,B_{\eta_0\ell_k'}^{\Pi_k})=\nu$ eventually.
		So the first equality in \eqref{eq:sameint} holds eventually, giving the desired contradiction.
	\end{proof}

	\begin{lemmaen}[$\bm{n(\cdot)=1}$ when $\bm{\tau=1}$]\label{sqrtsigma}
		Assume that $\Psi\in C^\infty(\bar B_{r}^2(z),\subman_{p,\ell})$ is a conformal immersion and $\Pi$ is a 2-plane with $\Psi(z+r\cdot)\in\mathcal{R}_{K_0,\delta_0}^\Pi$ and $\mz\int_{B_r^2(z)}\abs{\nabla\Psi}^2\le E$. If $\int_{B_1^2}\abs{A}^4\,d\operatorname{vol}_{g_\Psi}$ and $\ell$ are sufficiently small, then $\Pi\circ\Psi$ is a diffeomorphism from $\bar B_{s_0^2r}^2(z)$ onto its image.
	\end{lemmaen}

	\begin{proof}[Proof of Lemma \ref{sqrtsigma}]
		We can suppose $z=0$, $r=1$.
		Assume by contradiction that the claim does not hold, for a sequence of 2-planes $\Pi_k\to\Pi_\infty$ and immersions $\Psi_k:\bar B_1^2\to\subman_{p_k,\ell_k}$ with $\ell_k\to 0$ and second fundamental forms $A_k$ satisfying
		\begin{align} \int_{B_1^2}\abs{A_k}^4\,d\operatorname{vol}_{g_{\Psi_k}}\to 0. \end{align}
		Let $\lambda_k\in C^\infty(\bar B_1^2)$ be defined by $\abs{\de_1\Psi_k}=\abs{\de_2\Psi_k}=:e^{\lambda_k}$ and let $A_{p,\ell}$ and $\tilde A_k$ denote the second fundamental forms of $\subman_{p,\ell}\subseteq\R^q$ and of the immersion $\Psi_k$ in $\R^q$ respectively, so that $\tilde A_k=A_{p_k,\ell_k}+A_k$. Note that
		\begin{align} \norm{A_{p_k,\ell_k}}_{L^\infty}\le C(\subman)\ell_k\to 0, \end{align}
		so that
		\begin{align} \int_{B_1^2}\abs{\tilde A_k}^4\,d\operatorname{vol}_{g_{\Psi_k}}\to 0. \end{align}
		With a slight abuse of notation, let us drop the dependence on $k$ in the subsequent computations. We define the orthonormal frame
		\begin{align} \e_1:=e^{-\lambda}\de_1\Psi,\qquad\e_2:=e^{-\lambda}\de_2\Psi \end{align}
		for the tangent space of the immersed surface $\Psi$.
		It is straightforward to check that the map $\tilde e_1\wedge \tilde e_2:\bar B_1^2\to\Lambda_2\R^q$ has $\abs{\nabla(\tilde e_1\wedge \tilde e_2)}=e^\lambda\abs{\tilde A}$, so
		\begin{align} \int_{B_1^2}\abs{\nabla(\tilde e_1\wedge \tilde e_2)}^2\,d\mathcal{L}^2=\int_{B_1^2}e^{2\lambda}\abs{\tilde A}^2\,d\mathcal{L}^2=\int_{B_1^2}\abs{\tilde A}^2\,d\operatorname{vol}_{g_\Psi}\to 0 \end{align}
		by H\"older's inequality, since $\int_{B_1^2}\,d\operatorname{vol}_{g_\Psi}\le E$.
		We identify the Grassmannian $\text{Gr}(2,\R^q)$ of 2-planes in $\R^q$ with a submanifold of the projectivization of $\Lambda_2\R^q$, by means of Pl\"ucker's embedding. For $k$ large enough \cite[Lemma~5.1.4]{helein} applies and provides a rotated frame $(e_1,e_2)$, given by
		\begin{align} e_\C:=e_1+ie_2=e^{i\theta} \tilde e_\C,\qquad\tilde e_\C:=\e_1+i\e_2, \end{align}
		for a suitable real function $\theta\in W^{1,2}(B_1^2)$ minimizing $\int_{B_1^2}\abs{\nabla\theta+\e_1\cdot\nabla\e_2}^2$ (in particular, $\theta$ and $e_\C$ are smooth functions on $\bar B_1^2$) and with $\norm{\nabla e_\C}_{L^2}^2$ becoming arbitrarily small as $k\to\infty$. We will assume in the sequel that $\norm{\nabla e_\C}_{L^2}^2\le 1$.
		Observe that, whenever $\alpha,\beta\in C^1(\bar B_1^2)$,
		\begin{align*} \de_1\alpha\de_2\beta-\de_2\alpha\de_1\beta&=\frac{1}{4}(\de_1\alpha+\de_2\beta)^2+\frac{1}{4}(\de_2\alpha-\de_1\beta)^2-\frac{1}{4}(\de_1\alpha-\de_2\beta)^2-\frac{1}{4}(\de_2\alpha+\de_1\beta)^2 \\
		&=\abs{\de_z(\alpha+i\beta)}^2-\abs{\de_{\bar z}(\alpha+i\beta)}^2. \end{align*}
		Hence, being $\e_1+i\e_2=2e^{-\lambda}\de_{\bar z}\Psi$ and $\de_z\Psi\cdot\de_z\Psi=\de_{\bar z}\Psi\cdot\de_{\bar z}\Psi=0$ by conformality, we get
		\begin{align*} &-(\de_1\e_1\cdot\de_2\e_2-\de_2\e_1\cdot\de_1\e_2) \\
		&=4\de_{\bar z}(e^{-\lambda}\de_{\bar z}\Psi)\cdot\de_z(e^{-\lambda}\de_z\Psi)-4\de_{\bar z}(e^{-\lambda}\de_z\Psi)\cdot\de_z(e^{-\lambda}\de_{\bar z}\Psi) \\
		&=4e^{-2\lambda}(\de_{\bar z\bar z}^2\Psi\cdot\de_{zz}^2\Psi-\de_{\bar zz}^2\Psi\cdot\de_{\bar zz}^2\Psi-\de_{\bar z}\lambda\de_{\bar z}\Psi\cdot\de_{zz}^2\Psi-\de_z\lambda\de_z\Psi\cdot\de_{\bar z\bar z}^2\Psi) \\
		&\quad+2e^{-2\lambda}\de_{\bar z}\lambda\de_{\bar z}(\de_z\Psi\cdot\de_z\Psi)+2e^{-2\lambda}\de_z\lambda\de_z(\de_{\bar z}\Psi\cdot\de_{\bar z}\Psi) \\
		&=4e^{-2\lambda}(\de_{\bar z\bar z}^2\Psi\cdot\de_{zz}^2\Psi-\de_{\bar zz}^2\Psi\cdot\de_{\bar zz}^2\Psi-\de_{\bar z}\lambda\de_{\bar z}\Psi\cdot\de_{zz}^2\Psi-\de_z\lambda\de_z\Psi\cdot\de_{\bar z\bar z}^2\Psi). \end{align*}
		On the other hand we have
		\begin{align*} 2e^{2\lambda}\de_z\lambda=\de_z(e^{2\lambda})=\de_z(2\de_{\bar z}\Psi\cdot\de_z\Psi)=\de_{\bar z}(\de_z\Psi\cdot\de_z\Psi)+2\de_{\bar z}\Psi\cdot\de_{zz}^2\Psi=2\de_{\bar z}\Psi\cdot\de_{zz}^2\Psi, \end{align*}
		\begin{align*} \Delta(e^{2\lambda})&=4\de_{\bar zz}^2(2\de_{\bar z}\Psi\cdot\de_z\Psi)=8\de_{\bar z}(\de_{\bar z}\Psi\cdot\de_{zz}^2\Psi)+4\de_{\bar z\bar z}^2(\de_z\Psi\cdot\de_z\Psi) \\
		&=8(\de_{\bar z\bar z}^2\Psi\cdot\de_{zz}^2\Psi-\de_{\bar zz}^2\Psi\cdot\de_{\bar zz}^2\Psi)+4\de_{zz}^2(\de_{\bar z}\Psi\cdot\de_{\bar z}\Psi) \\
		&=8(\de_{\bar z\bar z}^2\Psi\cdot\de_{zz}^2\Psi-\de_{\bar zz}^2\Psi\cdot\de_{\bar zz}^2\Psi), \end{align*}
		so we arrive at
		\begin{align}\label{deltalambda} \de_1\e_1\cdot\de_2\e_2-\de_2\e_1\cdot\de_1\e_2
		=-\frac{\Delta(e^{2\lambda})}{2e^{2\lambda}}+8\de_{\bar z}\lambda\de_z\lambda
		=-\Delta\lambda. \end{align}
		Alternatively, since the projections of $\de_j\e_1$ and $\de_k\e_2$ onto the tangent space of the immersion $\Psi$ are orthogonal (being the projection of $\de_j\e_1$ a multiple of $\e_2$ and the projection of $\de_k\e_2$ a multiple of $\e_1$), we have
		\begin{align*} \de_1\e_1\cdot\de_2\e_2-\de_2\e_1\cdot\de_1\e_2=e^{2\lambda}(\tilde A(\e_1,\e_1)\cdot\tilde A(\e_2,\e_2)-\tilde A(\e_1,\e_2)\cdot\tilde A(\e_1,\e_2))=e^{2\lambda}K, \end{align*}
		by Gauss' formula, $K$ denoting the Gaussian curvature of the immersed surface. But, by the well-known formula for the curvature of a conformal metric, we have $K=-e^{-2\lambda}\Delta\lambda$, which gives again \eqref{deltalambda}. Moreover,
		\begin{align*} \de_1 e_1\cdot\de_2 e_2-\de_2 e_1\cdot\de_1 e_2
		&=\Im\ang{\nabla\bar{e_\C};\nabla e_\C}
		=\Im\ang{\nabla\bar{\tilde e_\C}-i\bar{\tilde e_\C}\otimes\nabla\theta;\nabla\tilde e_\C+i\tilde e_\C\otimes\nabla\theta} \\
		&=\Im\ang{\nabla\bar{\tilde e_\C};\nabla\tilde e_\C}
		=\de_1\e_1\cdot\de_2\e_2-\de_2\e_1\cdot\de_1\e_2, \end{align*}
		since $\ang{\bar{\tilde e_\C}\otimes\nabla\theta;\tilde e_\C\otimes\nabla\theta}$ is real and $\ang{-i\bar{\tilde e_\C}\otimes\nabla\theta;\nabla\tilde e_\C}=\bar{\ang{\nabla\bar{\tilde e_\C};i\tilde e_\C\otimes\nabla\theta}}$.
		Thus, calling $\mu\in C^\infty(\bar B_1^2)$ the solution to
		\begin{align*} \begin{cases}-\Delta\mu=\de_1 e_1\cdot\de_2 e_2-\de_2 e_1\cdot\de_1 e_2 & \text{on }B_1^2 \\ \ \mu=0 & \text{on }\de B_1^2, \end{cases} \end{align*}
		we obtain that $\lambda-\mu$ is harmonic and, by Wente's inequality,
		\begin{align}\label{wente} \norm{\mu}_{L^\infty}\le C(q)\pa{\norm{\nabla e_1}_{L^2}^2+\norm{\nabla e_2}_{L^2}^2}\le C(q). \end{align}
		Since $\lambda<e^{2\lambda}$, for all $x\in\bar B_{3/4}^2$ we get
		\begin{align}\label{upperbd} (\lambda-\mu)(x)=\media_{B_{1/4}^2(x)}(\lambda-\mu)\le\media_{B_{1/4}^2(x)}e^{2\lambda}+\norm{\mu}_{L^\infty}\le\frac{E}{\mathcal{L}^2(B_{1/4}^2)}+C(q). \end{align}
		Together with \eqref{wente}, this gives an upper bound for $\lambda$ on $B_{3/4}^2$, depending only on $E,q$. Although this is sufficient for the present purposes, one can also get a lower bound for $\lambda$ on $B_{s_0}^2$.
		Indeed, calling $M$ the right-hand side of \eqref{upperbd}, we obtain that $M-(\lambda-\mu)$ is a nonnegative harmonic function on $B_{3/4}^2$.
		Moreover, the length of the curve $\restr{\Psi}{\de B_{s_0}^2}$ is
		\begin{align} \int_{\de B_{s_0}^2}e^{\lambda}\ge 2\pi\eta_0, \end{align}
		since the composition of $\Pi\circ\restr{\Psi}{\de B_{s_0}^2}$ with the radial projection onto $\de B_{\eta_0}^\Pi$ (which does not increase the length) is surjective (being a generator of the fundamental group of $\de B_{\eta_0}^\Pi$). Hence, there exists some $x\in\de B_{s_0}^2$ such that $\lambda(x)\ge\log\pa{s_0^{-1}\eta_0}$. We deduce that
		\begin{align} \inf_{B_{s_0}^2}(M-(\lambda-\mu))\le M+C(q)-\log (s_0^{-1}\eta_0) \end{align}
		and so, by Harnack's inequality, the supremum of $M-(\lambda-\mu)$ on $B_{s_0}^2$ is bounded by a constant depending only on $E,s_0,\eta_0,q$.
		This, together with \eqref{upperbd} and \eqref{wente}, gives
		\begin{align} \norm{\lambda}_{L^\infty(B_{s_0}^2)}\le C(E,s_0,\eta_0,q). \end{align}
		The mean curvature of the immersion $\Psi$ is $\tilde H=\frac{1}{2e^{2\lambda}}(\tilde A(\de_1\Psi,\de_1\Psi)+\tilde A(\de_2\Psi,\de_2\Psi))=-\frac{\Delta\Psi}{2e^{2\lambda}}$ (note that $\Delta\Psi$ is already orthogonal to the tangent space of the immersion, since $\de_z\Psi\cdot\Delta\Psi=4\de_z\Psi\cdot\de_{\bar zz}^2\Psi=2\de_{\bar z}(\de_z\Psi\cdot\de_z\Psi)=0$). So we get
		\begin{align}\label{vanishing4}\begin{aligned} &\int_{B_{3/4}^2}\abs{\Delta\Psi_k}^4\,d\mathcal{L}^2=16\int_{B_{3/4}^2}\abs{\tilde H_k}^4 e^{6\lambda_k}\,d\operatorname{vol}_{g_{\Psi_k}} \\
		&\le C(E,q)\int_{B_{3/4}^2}\abs{\tilde A_k}^4\,d\operatorname{vol}_{g_{\Psi_k}}\to 0. \end{aligned}\end{align}
		Since $s_0\le\mz$, this implies that $(\Psi_k)$ is a bounded sequence in $W^{2,4}(B_{s_0}^2)$ (by Lemma \ref{easysobolev} applied to $\Psi_k(\frac{3}{4}\cdot)$), so by the compact embedding $W^{2,4}(B_{s_0}^2)\hookrightarrow C^1(\bar B_{s_0}^2)$ we obtain a strong limit $\Psi_\infty$ in $C^1(\bar B_{s_0}^2)$, up to subsequences. Thus $\Psi_\infty$ is weakly conformal and, by \eqref{vanishing4}, it is also harmonic. Lemma \ref{inj} applies (with $\Psi_\infty(s_0\cdot)$ and $\text{id}_{\R^2}$ in place of $\Psi$ and $\varphi$) and gives that $\Pi_\infty\circ\Psi_\infty$ is a diffeomorphism from $\bar B_{s_0/2}^2\supseteq\bar B_{s_0^2}^2$ onto its image, hence the same is eventually true for $\Pi_k\circ\Psi_k$, giving the desired contradiction.
	\end{proof}

	\section{Multiplicity one in the limit}\label{multonesec}
	
	\begin{thm}[$\bm{n(\cdot)=1}$ for small $\bm{\sigma^2\log(\sigma^{-1})\int|A|^4}$]\label{iteration}
		Assume $\Phi\in C^\infty(\bar B_{r}^2(z),\subman)$ is a conformal immersion, critical for \eqref{functional} on $B_{r}^2(z)$ and satisfying
		\begin{itemize}
			\item $\displaystyle\ell^{-1}(\Phi(z+r\cdot)-p)\in\mathcal{R}_{K_0,\delta_0}^\Pi$ for some $\sqrt{\sigma/\epsilon_0''}<\ell<1$ and $p\in\subman$,
			\item $\displaystyle\mz\int_{B_r^2(z)}\abs{\nabla\Phi}^2\le E_0\ell^2$,
			\item $\displaystyle\int_{\Phi^{-1}(B_\ell^q(p))}\,d\operatorname{vol}_{g_\Phi}\le V\pi\ell^2$ and $\displaystyle\int_{\Phi^{-1}(B_{\eta_0\ell}^q(p))}\,d\operatorname{vol}_{g_\Phi}\le V\pi(\eta_0\ell)^2$,
			\item $\displaystyle\sigma^2\log(\sigma^{-1})\int_{B_s^2}\abs{A}^4\,d\operatorname{vol}_{g_\Phi}\le\frac{\epsilon_0''}{E_0}\int_{B_s^2}\,d\operatorname{vol}_{g_\Phi}$ for all $0<s\le r$.
		\end{itemize}
		Then, if $\sigma$ and $\ell$ are small enough (independently of each other), we have
		\begin{align*}
			&n(\Phi-p,B_{s_0^2r}^2(z),B_{\eta_0\ell}^\Pi)=1.
		\end{align*}
	\end{thm}

	\begin{proof}[Proof of Theorem \ref{iteration}]
		Let $r_0:=r$, $p_0:=p$, $\ell_0:=\ell$, $\tau_0:=\sigma\ell_0^{-2}$ and $\Pi_0:=\Pi$. Note that
		\begin{align*}
			\Psi_0:=\ell^{-1}(\Phi-p)=\ell_0^{-1}(\Phi-p_0)
		\end{align*}
		is critical for \eqref{rescfunc}, with $\tau:=\tau_0\le\epsilon_0''$.
		Thus Lemma \ref{stabdist} applies to $\Psi_0$ (if $\ell$ is small enough), giving a new radius $\epsilon_0'r_0<r_1<s_0 r_0$, a new point $p'\in\subman_{p,\ell}$, a new scale $\epsilon_0'<\ell'<\frac{1}{2}$ and a new 2-plane $\Pi'$.
		Setting $r_1:=r'$, $p_1:=p_0+\ell_0 p'$, $\ell_1:=\ell'\ell_0$, $\tau_1:=\sigma\ell_1^{-2}$, $\Pi_1:=\Pi'$ and recalling \eqref{quanttrick}, the map
		\begin{align*}
			\Psi_1:=(\ell')^{-1}(\Psi_0-p')=\ell_1^{-1}(\Phi-p_1)
		\end{align*}
		still satisfies the hypotheses of Lemma \ref{stabdist}, with the parameters $r_1,\tau_1,p_1,\ell_1,\Pi_1$, provided that $\tau_1\le\epsilon_0'$: indeed, note that (assuming $\tau_1\le\epsilon_0'<1$)
		\begin{align*}
			&\tau_1^2\log(\tau_1^{-1})\int_{B_{r_1}^2(z)}\abs{A_{\Psi_1}}^4\,d\operatorname{vol}_{g_{\Psi_1}}
			\le\tau_1^2\log(\sigma^{-1})\int_{B_{r_1}^2(z)}\abs{A_{\Psi_1}}^4\,d\operatorname{vol}_{g_{\Psi_1}} \\
			&=\ell_1^{-2}\sigma^2\log(\sigma^{-1})\int_{B_{r_1}^2(z)}\abs{A_\Phi}^4\,d\operatorname{vol}_{g_\Phi}
			\le\frac{\epsilon_0''\ell_1^{-2}}{E_0}\int_{B_{r_1}^2(z)}\,d\operatorname{vol}_{g_\Phi}
			=\frac{\epsilon_0''}{2E_0}\int_{B_{r_1}^2(z)}\abs{\nabla\Psi_1}^2
			\le\epsilon_0''.
		\end{align*}
		Hence, we can iterate and define $r_j,\tau_j,p_j,\ell_j,\Pi_j$, for $j=0,1,\dots$, up to a maximum index $k\ge 1$ such that $\tau_j\le\epsilon_0''\le\epsilon_0'$ for $1\le j<k$ and $\tau_k>\epsilon_0''$: such $k$ exists since $\tau_j=\ell_j^{-2}\sigma\ge 4^j\tau_0$.
		With the same computation as above, this implies
		\begin{align}\label{sqrtsigma.a.est}
			\int_{B_{r_k}^2(z)}\abs{A}^4\,d\operatorname{vol}_{g_{\Psi_k}}\le\frac{\epsilon_0''}{\tau_k^2\log(\sigma^{-1})}\le\frac{1}{\epsilon_0''\log(\sigma^{-1})}.
		\end{align}
		If $\sigma$ and $\ell$ are small enough, Lemma \ref{sqrtsigma} applies to the map
		$\Psi_k:=\ell_k^{-1}(\Psi-p_k)$, on the ball $B_{r_k}^2(z)$: indeed, note that $\ell_k\le\ell$ and $\int_{B_{r_k}^2(z)}|A|^4\,d\operatorname{vol}_{g_{\Psi_k}}$ can be assumed arbitrarily small (by taking $\sigma$ and $\ell$ small enough), by virtue of \eqref{sqrtsigma.a.est}. This, together with Lemma \ref{winded}, gives
		\begin{align*}
			&n(\Psi_k,B_{s_0^2r_k}^2(z),B_{\eta_0}^{\Pi_k})=1.
		\end{align*}
		Also, Lemma \ref{smallcontrib} applies for all $j=0,\dots,k-1$, giving
		\begin{align*}
			n(\Phi-p,B_{s_0^2r}^2(z),B_{\eta_0\ell}^\Pi)
			&=n(\Psi_0,B_{s_0^2r_0}^2(z),B_{\eta_0}^{\Pi_0}) \\
			&=n(\Psi_1,B_{s_0^2r_1}^2(z),B_{\eta_0}^{\Pi_1}) \\
			&=\cdots \\
			&=n(\Psi_k,B_{s_0^2r_k}^2(z),B_{\eta_0}^{\Pi_k}) \\
			&=1. \qedhere
		\end{align*}
	\end{proof}
	
	As in Section \ref{backsec}, assume now that $\Phi_k:\Sigma\to\subman$ is a sequence of critical points for
	\begin{align} \int_\Sigma\,d\operatorname{vol}_{g_{\Phi_k}}+\sigma_k^2\int_\Sigma(1+|A|^2)^2\,d\operatorname{vol}_{g_{\Phi_k}} \end{align}
	with controlled area, namely
	\begin{align*} \lambda\le\int_\Sigma\,d\operatorname{vol}_{g_{\Phi_k}}\le\Lambda, \end{align*}
	and with \begin{align*} \sigma_k\to 0,\qquad\sigma_k^2\log(\sigma_k^{-1})\int_\Sigma(1+|A|^2)^2\,d\operatorname{vol}_{g_{\Phi_k}}\to 0. \end{align*}
	
	By the main result of \cite{rivminmax}, up to subsequences the varifolds $\vfd_k$ induced by $\Phi_k$ converge to a parametrized stationary varifold (see e.g. \cite[Definition~2.2]{pigriv} and \cite[Remark~2.3]{pigriv} for two equivalent definitions of this notion).
	
	As explained in \cite{rivminmax}, there could be bubbling points, and also the conformal structures induced by $\Phi_k$ could degenerate (in the space of conformal structures up to diffeomorphisms). The latter cannot happen if $\Sigma$ is a sphere, while it can yield a cylinder $\C/\Z$ (or, equivalently, $\C\setminus\{0\}=\hat\C\setminus\{0,\infty\}$) in the limit when $\Sigma$ is a torus. While both of these statements are easy consequences of the uniformization theorem for Riemann surfaces, if the genus is at least two then the general picture of degenerating conformal structures is more complicated; we refer the reader to \cite[Section~IV.5]{hum} for its description.\footnote{Given a sequence of closed hyperbolic surfaces, we can assume that they have been decomposed into marked pairs of pants, with a constant combinatorial configuration (see e.g. \cite[Section~IV.3]{hum}). We can then cut them open along the marked geodesics whose length converges to zero and apply \cite[Proposition~IV.5.1]{hum} to the resulting surfaces with boundary.}
	
	In the remainder of the paper, we will assume for simplicity that there is no bubbling and no degeneration of the conformal structure. 	Note that the arguments will apply also to the general case, working on appropriate domains different from $\Sigma$.
	
	Up to precomposing $\Phi_k$ with suitable diffeomorphisms of $\Sigma$, we can thus assume that there exist metrics $g_k$ of constant curvature ($1$, $0$ or $-1$, depending on the genus of $\Sigma$) such that $\Phi_k:(\Sigma,g_k)\to\subman$ is conformal, and such that $g_k$ converges smoothly to a limiting Riemannian metric $g_\infty$.
	The limiting varifold $\vfd_\infty$ is a parametrized stationary varifold,
	of the form $(\Sigma_\infty,\Theta_\infty,N_\infty)$, where $\Theta_\infty:\Sigma_\infty\to\subman$ is a smooth branched minimal immersion. Also, calling $\Phi_\infty\in W^{1,2}(\Sigma,\subman)$ the weak limit of $\Phi_k$ (up to subsequences), $\Theta_\infty=\Phi_\infty\circ\varphi_\infty^{-1}$ for some (locally) quasiconformal homeomorphism $\varphi_\infty:\Sigma\to\Sigma_\infty$.\footnote{Here $\Sigma_\infty$ is a Riemann surface homeomorphic (and thus diffeomorphic) to $\Sigma$.} In particular, $\Phi_\infty$ is continuous.
	
	In local conformal coordinates for $(\Sigma,g_\infty)$, as in \eqref{blackboxmeas} we have
	\begin{align}\label{structn} \operatorname{vol}_{g_{\Phi_k}}\weakstarto N_\infty\abs{\de_1\Phi_\infty\wedge\de_2\Phi_\infty}\,\mathcal{L}^2\ge\mz|\nabla\Phi_\infty|^2\,\mathcal{L}^2. \end{align}
	By the regularity result of \cite{pigriv}, which was already exploited in Section \ref{coresec}, $N_\infty$ is locally a.e. constant and thus a.e. constant (being $\Sigma_\infty$ connected).

	Setting $\nu_k:=\operatorname{vol}_{g_{\Phi_k}}$ and $\nu_\infty:=N_\infty|\de_1\Phi_\infty\wedge\de_2\Phi_\infty|\,\mathcal{L}^2$ (in local conformal coordinates for $\Sigma$), by \eqref{structn} we have $\nu_k\weakstarto\nu_\infty$.
	We can find a conformal disk $U\subset(\Sigma,g_\infty)$, which we identify with $B_1^2\subset\C$ and fix in the sequel, such that $\nu_\infty(B_{1/2}^2)>0$.

	\begin{definition}
		We denote by $\nu$ the constant value of $N_\infty$. Also, we call $T$ the set of bad points $z\in B_1^2$ which are not Lebesgue for $\nabla\Phi_\infty$, or such that $\nabla\Phi_\infty(z)$ does not have full rank, or such that
		\begin{align}\label{bigdistort} \max_{\abs{x}=1}\abs{\ang{\nabla\Phi_\infty(z),x}}>2\nu\min_{\abs{x}=1}\abs{\ang{\nabla\Phi_\infty(z),x}}. \end{align}
		We have $\mathcal{L}^2(T)=0$, since $\nabla\Theta_\infty$ has full rank a.e. by conformality (hence the same holds for $\Phi_\infty$ by the chain rule\footnote{See e.g. \cite[Lemma~4.12]{imayoshi} and \cite[Lemma~III.6.4]{lehto}.}) and since \eqref{bigdistort} implies
		$\nu|\de_1\Phi_\infty\wedge\de_2\Phi_\infty|(z)<\mz|\nabla\Phi_\infty|^2(z)$ (as it can be immediately verified using a singular value decomposition for $\nabla\Phi_\infty(z)$).
	\end{definition}

	\begin{definition}
		We now specify $K'':=2\nu$ and we set $E'':=\pi\nu((K'')^2+1)$.
		Finally, we choose $V>0$ such that $V=\floor{V}+\mz$ and
		\begin{align}\label{maxeasy}
		&\norm{\vfd_\infty}(\bar B_\ell^q(p))<V\pi\ell^2
		\end{align}
		for all $\ell>0$ and all $p\in\subman$. Such $V$ exists by the monotonicity formula satisfied by the stationary varifold $\vfd_\infty$.
		Note that now also the constants $K_0$, $E_0$, $s_0$, $\eta_0$, as well as $\epsilon_0$, $\delta_0$, $\epsilon_0'$ and $\epsilon_0''$, are determined.
	\end{definition}

	\begin{thm}[multiplicity one]\label{main.bis}
		We have $N_\infty=1$ a.e., or equivalently $\nu=1$.
	\end{thm}

	\begin{proof}[Proof of Theorem \ref{main.bis}]
		Let $\bad_k$ be the Borel set of points $z\in B_{1/2}^2$ such that
		\begin{align*}
			&\sigma_k^2\log(\sigma_k^{-1})\int_{B_{r}^2(z)}\abs{A}^4\,d\operatorname{vol}_{g_{\Phi_k}}\ge\frac{\epsilon_0''}{E_0}\int_{B_{r}^2(z)}\,d\operatorname{vol}_{g_{\Phi_k}}
		\end{align*}
		for some radius $0<r<\mz$. By Besicovitch's covering lemma, we get a collection of points $z_i\in\bad_k$ and radii $0<r_i<\mz$ such that
		\begin{align*} \sigma_k^2\log(\sigma_k^{-1})\int_{B_{r_i}^2(z_i)}\abs{A}^4\,d\operatorname{vol}_{g_{\Phi_k}}\ge\frac{\epsilon_0''}{E_0}\int_{B_{r_i}^2(z_i)}\,d\operatorname{vol}_{g_{\Phi_k}},\qquad\uno_{\bad_k}\le\sum_i\uno_{B_{r_i}^2(z_i)}\le\mathfrak{N}, \end{align*}
		for some universal constant $\mathfrak{N}$. Thus we get
		\begin{align*} \nu_k(\bad_k)&\le\sum_i\operatorname{vol}_{g_{\Phi_k}}(B_{r_i}^2(z_i))\le\frac{E_0}{\epsilon_0''}\sigma_k^2\log(\sigma_k^{-1})\sum_i\int_{B_{r_i}^2(z_i)}\abs{A}^4\,d\operatorname{vol}_{g_{\Phi_k}} \\
		&\le\frac{E_0\mathfrak{N}}{\epsilon_0''}\sigma_k^2\log(\sigma_k^{-1})\int_\Sigma\abs{A}^4\,d\operatorname{vol}_{g_{\Phi_k}}\to 0. \end{align*}
		Up to subsequences, we can assume that $\bar{B_{1/2}^2\setminus\bad_k}$ converges in the Hausdorff topology to some compact set $S\subseteq\bar B_{1/2}^2$. We remark that $\nu_\infty(S)>0$: indeed, for any compact neighborhood $F$ of $S$ in $B_1^2$, we have $B_{1/2}^2\setminus\bad_k\subseteq F$ eventually and so
		\begin{align*}
			&\nu_\infty(F)\ge\limsup_{k\to\infty}\nu_k(F)\ge\limsup_{k\to\infty}(\nu_k(B_{1/2}^2)-\nu_k(\bad_k))=\limsup_{k\to\infty}\nu_k(B_{1/2}^2)\ge\nu_\infty(B_{1/2}^2)>0.
		\end{align*}
		It follows from \eqref{structn} that $\mathcal{L}^2(S)>0$.
		
		We now show that $N_\infty=1$ a.e. on $S\setminus T$, which has positive Lebesgue measure.
		This will show that $\nu=1$, as desired. Fix any $z\in S\setminus T$ and take a sequence $z_k\in B_{1/2}^2\setminus\bad_k$ with $z_k\to z$. Locally we can find conformal reparametrizations $\tilde\Phi_k$ of $\Phi_k(z_k+\cdot)$, by means of diffeomorphisms converging smoothly to the identity on a small neighborhood of $0$.\footnote{For instance, one can isometrically identify a neighborhood of $z_k$ in $(\Sigma,g_k)$ with a neighborhood of $z$ in $(\Sigma,g_\infty)$, by means of the exponential map.} By weak convergence $\tilde\Phi_k\weakto\Phi_\infty(z+\cdot)$ in $W^{1,2}$, for a.e. radius $r>0$ we have
		\begin{align}\label{slicefinal} \quad\tilde\Phi_k(r\cdot)\to\Phi_\infty(z+r\cdot)\quad\text{in }C^0(\de B_1^2\cup\de B_{s_0}^2\cup\de B_{s_0^2}^2) \end{align}
		up to further subsequences.\footnote{This can be obtained by applying \cite[Lemma~A.5]{pigriv} to the weakly converging $\R^{3\envdim}$-valued maps
		\begin{align*}
			&(\tilde\Phi_k,\tilde\Phi_k(s_0\cdot),\tilde\Phi_k(s_0^2\cdot))\weakto(\Phi_\infty(z+\cdot),\Phi_\infty(z+s_0\cdot),\Phi_\infty(z+s_0^2\cdot)).
		\end{align*}}
		Using \cite[Lemma~A.4]{pigriv} and the fact that $z\nin T$, we can assume that $r$ satisfies
		\begin{align} \abs{\Phi_\infty(z+rx)-\Phi_\infty(z)-\ang{\nabla\Phi_\infty(z),rx}}<\delta_0\ell\quad\text{for }x\in\de B_1^2\cup\de B_{s_0}^2\cup\de B_{s_0^2}^2, \end{align}
		\begin{align}\label{smallasen} \mz\int_{B_{r}^2(z)}\abs{\nabla\Phi_\infty}^2<(\pi r^2)\abs{\nabla\Phi_\infty(z)}^2\le \ell^2\pi((K'')^2+1), \end{align}
		with $\ell:=r\min_{\abs{x}=1}\abs{\ang{\nabla\Phi_\infty(0),x}}$.
		Setting $p:=\Phi_\infty(z)$, note that \eqref{maxeasy} gives
		\begin{align*}
			&\norm{\vfd_k}(B_\ell^q(p))<V\pi\ell^2,\quad\norm{\vfd_k}(B_{\eta_0\ell}^q(p))<V\pi(\eta_0\ell)^2
		\end{align*}
		eventually, which trivially implies
		\begin{align}
			&\int_{\tilde\Phi_k^{-1}(B_\ell^q(p))}\,d\operatorname{vol}_{g_{\tilde\Phi_k}}<V\pi\ell^2,\quad\int_{\tilde\Phi_k^{-1}(B_{\eta_0\ell}^q(p))}\,d\operatorname{vol}_{g_{\tilde\Phi_k}}<V\pi(\eta_0\ell)^2.
		\end{align}
		Also, \eqref{structn} and \eqref{smallasen} give
		\begin{align*}
			&\lim_{k\to\infty}\mz\int_{B_{r}^2(z)}\abs{\nabla\tilde\Phi_k}^2
			=\lim_{k\to\infty}\nu_k(B_r^2(z))\le\frac{\nu}{2}\int_{B_r^2(z)}\abs{\nabla\Phi_\infty}^2<E''\ell^2.
		\end{align*}
		Thanks to the fact that $z_k\nin\bad_k$ and the above inequalities, eventually $\tilde\Phi_k$ satisfies the hypotheses of Theorem \ref{iteration} on the ball $B_r^2$, provided that $r$ (and thus $\ell$) is chosen small enough. Setting $\Psi_k:=\ell^{-1}(\tilde\Phi_k-p)$, we infer that
		\begin{align}\label{multone.macro}
			&n(\Psi_k,B_{s_0^2r}^2,B_{\eta_0}^\Pi)=1,
		\end{align}
		where $\Pi$ is the 2-plane spanned by $\nabla\Phi_\infty(0)$.
		
		Since $r$ can be chosen arbitrarily small (possibly changing the subsequence guaranteeing \eqref{slicefinal}), the argument used in the proof of \cite[Lemma~III.10]{rivminmax} shows that $N_\infty(z)=1$.
		
		Alternatively, \eqref{multone.macro} gives
		\begin{align*}
			\abs{\frac{\norm{\Pi_*\vfd_k'}(B_{\eta_0}^\Pi)}{\pi\eta_0^2}-1}<\frac{1}{8},
		\end{align*}
		where the varifold $\vfd_k'$ is induced by $\restr{\Psi_k}{B_{s_0^2r}^2}$ and converges to the varifold $\vfd_\infty'$ induced by $(\varphi_\infty(B_{s_0^2r}^2(z)),\ell^{-1}(\Theta_\infty-p),\nu)$. Assuming without loss of generality that $\nabla\Theta_\infty(\varphi_\infty(z))\neq 0$, $\Pi\circ\Theta_\infty$ is a diffeomorphism from $\varphi_\infty(B_{s_0^2r}^2(z))$ onto its image (for $r$ small enough). Hence, at a.e. point of $\Pi$ the varifold $\Pi_*\vfd_\infty$ has density either $0$ or $\nu$. Since $\Pi\circ\Phi_\infty(B_{s_0^2r}^2(z))$ is a superset of $B_{\eta_0\ell}^\Pi(\Pi(p))$ (by Lemma \ref{winded}), it follows that
		\begin{align*}
			&\frac{\norm{\Pi_*\vfd_\infty'}(B_{\eta_0}^\Pi)}{\pi\eta_0^2}=\nu.
		\end{align*}
		The convergence $\frac{\norm{\Pi_*\vfd_k'}(B_{\eta_0}^\Pi)}{\pi\eta_0^2}\to\frac{\norm{\Pi_*\vfd_\infty'}(B_{\eta_0}^\Pi)}{\pi\eta_0^2}$ thus gives
		$\abs{\nu-1}\le\frac{1}{8}$, and again we conclude that $\nu=1$.
	\end{proof}
	
	
	\appendix
	\section*{Appendix.}
	\renewcommand{\thesection}{A}
	\setcounter{definition}{0}
	\setcounter{equation}{0}
	
	\begin{lemmaen}[big image under a boundary constraint]\label{winded}
		Assume that $F\in C^0(\bar B_1^2,\R^2)$ satisfies
		\begin{align} \abs{F(x)-\varphi(x)}\le\delta\qquad\text{for all }x\in\de B_1^2 \end{align}
		for some $0<\delta<1$ and some homeomorphism $\varphi:\R^2\to\R^2$, with $\varphi(0)=0$ and $\min_{\abs{x}=1}\abs{\varphi(x)}\ge 1$. Then
		\begin{align} F(B_1^2)\supseteq B_{1-\delta}^2. \end{align}
	\end{lemmaen}

	\begin{proof}[Proof of Lemma \ref{winded}]
		It suffices to show that, for a fixed $y\in B_{1-\delta}^2$, the closed curve $\Gamma':=\restr{F}{\de B_1^2}$ is not contractible in $\R^2\setminus\set{y}$: once this is done, if we had $y\nin F(B_1^2)$, i.e. $y\nin F(\bar B_1^2)$, then $F$ would provide a homotopy from $\Gamma'$ to the constant curve $F(0)$ in $\R^2\setminus\set{y}$, yielding a contradiction.
		
		Letting $\Gamma:=\restr{\varphi}{\de B_1^2}$ and $\gamma:=\Gamma'-\Gamma$, we have $\abs{\gamma(x)}\le\delta$ for all $x\in\de B_1^2$. Hence, $\Gamma$ is homotopic to $\Gamma'$ in $\R^2\setminus B_{1-\delta}^2\subseteq\R^2\setminus\set{y}$ by means of the homotopy
		\begin{align*} \Gamma+t\gamma,\qquad 0\le t\le 1. \end{align*}
		So we are left to show that $\Gamma$ is not contractible in $\R^2\setminus\set{y}$, i.e. that $\Gamma-y$ is not contractible in $\R^2\setminus\set{0}$. The curve $\Gamma-y$ is homotopic to $\Gamma$ in $\R^2\setminus\set{0}$, by means of the homotopy
		\begin{align*} \Gamma-ty,\qquad 0\le t\le 1, \end{align*}
		which avoids the origin since $\abs{y}<1$. Finally, $\Gamma$ is not contractible in $\R^2\setminus\set{0}$, since $\varphi$ (once restricted to a homeomorphism of $\R^2\setminus\set{0}$) induces an automorphism of $\pi_1(\R^2\setminus\set{0})$ sending the class of the generator $\text{id}_{\de B_1^2}$ to the class of $\Gamma$. Hence, $\Gamma-y$ is not contractible in $\R^2\setminus\set{0}$, too, as desired.
	\end{proof}

%
	
	\begin{lemmaen}[elliptic $\bm{W^{2,4}}$-estimate]\label{easysobolev}
		For a function $\Psi\in C^\infty(\bar B_1)$ and a $0<\tau<1$ we have
		\begin{align*} \norm{\Psi}_{W^{2,4}(B_{\tau}^2)}\le C(\tau)(\norm{\Delta\Psi}_{L^4(B_1^2)}+\norm{\nabla\Psi}_{L^2(B_1^2)}+\norm{\Psi}_{L^2(B_1^2)}). \end{align*}
	\end{lemmaen}
	
	\begin{proof}[Proof of Lemma \ref{easysobolev}]
		Given two radii $0<r<s\le 1$, let us choose a cut-off function $\rho\in C^\infty_c(B_s^2)$ with $\rho=1$ on $B_r^2$. Since $\rho\Psi\in C^\infty_c(\R^2)$, standard Calder\'on--Zygmund estimates give
		\begin{align}\label{cutoff} \begin{aligned} \norm{\nabla^2\Psi}_{L^p(B_r^2)}
		&\le\norm{\nabla^2(\rho\Psi)}_{L^p(\R^2)}
		\le C(p)\norm{\Delta(\rho\Psi)}_{L^p(\R^2)} \\
		&\le C(p,r,s)(\norm{\Delta\Psi}_{L^p(B_s^2)}+\norm{\nabla\Psi}_{L^p(B_s^2)}+\norm{\Psi}_{L^p(B_s^2)}) \end{aligned} \end{align}
		for all $1<p<\infty$.
		Setting $t:=\frac{1+\tau}{2}$ and applying \eqref{cutoff} with $p:=2$, $r:=t$ and $s:=1$ we get
		\begin{align*} \norm{\nabla^2\Psi}_{L^2(B_t^2)}\le C(\tau)(\norm{\Delta\Psi}_{L^2(B_1^2)}+\norm{\nabla\Psi}_{L^2(B_1^2)}+\norm{\Psi}_{L^2(B_1^2)}), \end{align*}
		hence $\norm{\Psi}_{W^{2,2}(B_t^2)}$ is bounded by the desired quantity. Using Sobolev's embedding $W^{2,2}(B_t^2)\hookrightarrow W^{1,4}(B_t^2)$ and \eqref{cutoff} with $p:=4$, $r:=\tau$ and $s:=t$, we obtain
		\begin{align*} \norm{\Psi}_{W^{2,4}(B_{\tau}^2)}&\le C(\tau)(\norm{\Delta\Psi}_{L^4(B_{t}^2)}+\norm{\Psi}_{W^{2,2}(B_{t}^2)}) \\
		&\le C(\tau)(\norm{\Delta\Psi}_{L^4(B_1^2)}+\norm{\nabla\Psi}_{L^2(B_1^2)}+\norm{\Psi}_{L^2(B_1^2)}). \qedhere \end{align*}
	\end{proof}

	\begin{lemmaen}[compactness of normal solutions to Beltrami equation]\label{qchomcptaux}
		Given a sequence $\psi_k:\C\to\C$ of $K$-quasiconformal homeomorphisms with the normalization conditions
		\begin{align*}
		\psi_k(0)=0,\qquad\psi_k(1)=1,
		\end{align*}
		there exists a $K$-quasiconformal homeomorphism $\psi_\infty:\C\to\C$ satisfying the same normalization condition and such that, up to subsequences,
		$\psi_k\to\psi_\infty$ and $\psi_k^{-1}\to\psi_\infty^{-1}$ in $C^0_{loc}(\C)$.
	\end{lemmaen}
	
	\begin{proof}[Proof of Lemma \ref{qchomcptaux}]
		Let $\mu_k\in\mathcal{E}_K$ be defined by $\de_{\bar z}\psi_k=\mu_k\de_{z}\psi_k$.\footnote{Actually, the coefficient $\mu_k$ is uniquely determined a.e., as $\de_z\psi_k\neq 0$ a.e. (this follows from $\operatorname{id}_\C=\psi_k^{-1}\circ\psi_k$ and the chain rule \cite[Lemma~III.6.4]{lehto}, together with $\abs{\de_{\bar z}\psi_k}\le\abs{\de_z\psi_k}$).} Existence and uniqueness of a $K$-quasiconformal homeomorphism satisfying this equation and the normalization conditions is shown in \cite[Theorem~4.30]{imayoshi}.
		
		Given $M>0$, we consider the set $\mathcal{E}_K^M:=\set{\mu\in\mathcal{E}_K:\mu=0\text{ a.e. on }\C\setminus B_M^2}$. If $F^\mu$ denotes the normal solution to the equation $\de_{\bar z}F^\mu=\mu\de_z F^\mu$ (in the sense of \cite[Theorem~4.24]{imayoshi}), then $F^\mu$ satisfies estimates (4.21) and (4.24) in \cite{imayoshi}.
		Applying them with the points $0$ and $1$, we infer that also the map $f^\mu:=F^\mu(1)^{-1}F^\mu$ satisfies estimates of the form
		\begin{align}\label{hol1} \abs{f^\mu(z_1)-f^\mu(z_2)}\le C\abs{z_1-z_2}^\alpha+C\abs{z_1-z_2}, \end{align}
		\begin{align}\label{hol2} \abs{z_1-z_2}\le C\abs{f^\mu(z_1)-f^\mu(z_2)}^\alpha+C\abs{f^\mu(z_1)-f^\mu(z_2)}, \end{align}
		with $C$ and $0<\alpha<1$ depending only on $K$ and $M$. Given a sequence of homeomorphisms $f_k:\C\to\C$ satisfying these estimates, Ascoli--Arzel\`a theorem applies to $f_k$ and $f_k^{-1}$ and so we can extract a subsequence (not relabeled) such that
		\begin{align*} f_k\to f_\infty, \quad f_k^{-1}\to\tilde  f_\infty\quad\text{in }C^0_{loc}(\C). \end{align*}
		From
		$f_k^{-1}\circ f_k=f_k\circ f_k^{-1}=\operatorname{id}_{\C}$ we get $\tilde f_\infty\circ f_\infty=f_\infty\circ\tilde f_\infty=\operatorname{id}_{\C}$ and thus $f_\infty:\C\to\C$ is a homeomorphism, with $\tilde f_\infty=f_\infty^{-1}$. Also, since $f_k(z),f_k^{-1}(z)\to\infty$ uniformly as $z\to\infty$, we deduce that the canonical extensions $\widehat{f}_k:\widehat\C\to\widehat\C$ converge uniformly to $\widehat{f}_\infty$ and that the same holds for $\widehat{f}_k^{-1}$.
		
		We now closely examine the proof of \cite[Theorem~4.30]{imayoshi}: let $\tilde\mu_k\in\mathcal{E}_K^1$ be given by equation (4.25) in \cite{imayoshi}, with $\mu_k\uno_{\C\setminus B_1^2}$ in place of $\mu$, and
		\begin{align*}
		g_k:\widehat\C\to\widehat\C,\qquad g_k(z):=\widehat{f^{\tilde\mu_k}}(z^{-1})^{-1}.
		\end{align*}
		This map corresponds to the map $f^{\mu_1}$ in the aforementioned proof (with $\mu_k$ in place of $\mu$). The lower bound \eqref{hol2}, applied with $f^{\tilde\mu_k}$ and $z_1:=f^{\tilde\mu_k}(z^{-1})$, $z_2:=0$, shows that $\abs{g_k(z)}$ is bounded above by some $M$, for all $k$ and all $z\in\bar B_1^2$. Hence, defining $\mu_{k,2}$ as in equation (4.27) in \cite{imayoshi} (with $\mu_k$ in place of $\mu$), we get $\mu_{k,2}\in\mathcal{E}_{K}^M$. Calling $h_k:\widehat\C\to\widehat\C$ the associated quasiconformal homeomorphism, normalized so that $h_k(0)=0$ and $h_k(1)=1$, by the above argument we obtain the uniform convergence
		\begin{align*}
		g_k\to g_\infty,\quad g_k^{-1}\to g_\infty^{-1},\quad h_k\to h_\infty,\quad h_k^{-1}\to h_\infty^{-1}
		\end{align*}
		up to subsequences, for suitable homeomorphisms $g_\infty$ and $h_\infty$ of the Riemann sphere $\widehat\C$. Setting $\psi_\infty:=\restr{h_\infty\circ g_\infty}{\C}$ and observing that $\psi_k=\restr{h_k\circ g_k}{\C}$, we get the desired convergence $\psi_k\to\psi_\infty$ and $\psi_k^{-1}\to\psi_\infty^{-1}$ in $C^0_{loc}(\C)$.
		
		Finally, we show that $\psi_\infty$ is a $K$-quasiconformal homeomorphism. Given an open rectangle $R\cptsub\C$, \cite[Lemma~4.12]{imayoshi} gives
		\begin{align*} \mathcal{L}^2(\psi_k(R))=\int_R(\abs{\de_z\psi_k}^2-\abs{\de_{\bar z}\psi_k}^2)\ge\int_R(1-k^2)\abs{\de_z\psi_k}^2\ge (1-k^2)k^{-2}\int_R\abs{\de_{\bar z}\psi_k}^2, \end{align*}
		where $k:=\frac{K-1}{K+1}$. Since $\mathcal{L}^2(\psi_k(R))\to\mathcal{L}^2(\psi_\infty(R))$, we deduce that $\psi_k$ is bounded in $W^{1,2}(R)$, thus $\psi_\infty$ is the limit of $\psi_k$ in the weak $W^{1,2}_{loc}(\C)$-topology.
		Given $\rho,\psi^1,\psi^2\in C^\infty_c(\C)$, integration by parts shows that
		\begin{align}\label{detweak}
		\int_\C\rho(\de_1\psi^1\de_2\psi^2-\de_2\psi^1\de_1\psi^2)
		=-\int_\C(\de_1\rho\psi^1\de_2\psi^2-\de_2\rho\psi^1\de_1\psi^2).
		\end{align}
		Writing $\psi_k=\varphi_k^1+i\psi_k^2$, a standard density argument shows that \eqref{detweak} still holds with $\psi^1,\psi^2$ replaced by $\psi_k^1,\psi_k^2$, for $k\in\N\cup\set{\infty}$. Hence, observing that $\abs{\de_z\psi_k}^2-\abs{\de_{\bar z}\psi_k}^2=(\de_1\psi_k^1\de_2\psi_k^2-\de_2\psi_k^1\de_1\psi_k^2)$, we get
		\begin{align}\label{convdet} \int_\C\rho(\abs{\de_z\psi_k}^2-\abs{\de_{\bar z}\psi_k}^2)\to\int_\C\rho(\abs{\de_z\psi_\infty}^2-\abs{\de_{\bar z}\psi_\infty}^2). \end{align}
		Defining the positive measures $\nu_k:=(\abs{\de_z\psi_k}^2-\abs{\de_{\bar z}\psi_k}^2)\mathcal{L}^2$, up to further subsequences we can assume that $\nu_k\weakstarto\nu_\infty$ as Radon measures. For any rectangle $R$ such that $\nu_\infty(\de R)=0$, approximating $\uno_R$ from above and below with smooth functions and applying \eqref{convdet} we get
		\begin{align*} \int_R(\abs{\de_z\psi_k}^2-\abs{\de_{\bar z}\psi_k}^2)\to\int_R(\abs{\de_z\psi_\infty}^2-\abs{\de_{\bar z}\psi_\infty}^2). \end{align*}
		By monotonicity of both sides, this actually holds for every rectangle $R$. On the other hand, by lower semicontinuity of the $L^2$-norm,
		\begin{align*} \int_R(1-k^2)\abs{\de_z\psi_\infty}^2&\le\liminf_{k\to\infty}\int_R(1-k^2)\abs{\de_z\psi_k}^2\le\lim_{k\to\infty}\int_R(\abs{\de_z\psi_k}^2-\abs{\de_{\bar z}\psi_k}^2) \\
		&=\int_R(\abs{\de_z\psi_\infty}^2-\abs{\de_{\bar z}\psi_\infty}^2). \end{align*}
		Since $R$ is arbitrary, we get $\abs{\de_{\bar z}\psi_\infty}\le k\abs{\de_z\psi_\infty}$ a.e., as desired.
	\end{proof}		
	
	\begin{corollary}[compactness of $\bm{\mathcal{D}_K}$]\label{qchomcpt}
		Given a sequence $\varphi_k\in\mathcal{D}_K$, there exists $\varphi_\infty\in\mathcal{D}_K$ such that, up to subsequences,
		$\varphi_k\to\varphi_\infty$ and $\varphi_k^{-1}\to\varphi_\infty^{-1}$ in $C^0_{loc}(\C)$.
	\end{corollary}
	
	\begin{proof}[Proof of Corollary \ref{qchomcpt}]
		Let $\mu_k\in\mathcal{E}_K$ be defined by $\de_{\bar z}\varphi_k=\mu_k\de_{z}\varphi_k$ for all $k$ and let $\psi_k:\C\to\C$ be the unique $K$-quasiconformal homeomorphism satisfying the same differential equation, as well as $\psi_k(0)=0$, $\psi_k(1)=1$ (see \cite[Theorem~4.30]{imayoshi}).
		
		By Lemma \ref{qchomcptaux}, up to subsequences there exists a $K$-quasiconformal homeomorphism $\psi_\infty$ such that $\psi_k\to\psi_\infty$ and $\psi_k^{-1}\to\psi_\infty^{-1}$ in $C^0_{loc}(\C)$.
		
		By the chain rule (see \cite[Lemma~III.6.4]{lehto}), the map $\psi_k\circ\varphi_k^{-1}:\C\to\C$ is a biholomorphism and fixes the origin, so it equals the multiplication by a nonzero complex number $\lambda_k$, i.e. $\psi_k=\lambda_k\varphi_k$. On the other hand,
		\begin{align*}
		\abs{\lambda_k}=\min_{x\in\de B_1^2}\abs{\psi_k(x)}\to\min_{x\in\de B_1^2}\abs{\psi_\infty(x)}\in(0,\infty).
		\end{align*}
		Hence, up to further subsequences we can suppose that $\lambda_k\to\lambda_\infty\in\C\setminus\set{0}$. The statement follows with $\varphi_\infty:=\lambda_\infty^{-1}\psi_\infty$.
	\end{proof}

	\begin{rmk}
		In general, given $\varphi_k\in\mathcal{D}_K$ (for $k\in\N\cup\set{\infty}$) with $\varphi_k\to\varphi_\infty$ and $\varphi_k^{-1}\to\varphi_\infty^{-1}$ locally uniformly, it is not true that the corresponding Beltrami coefficients satisfy $\mu_k\weakstarto\mu_\infty$ in $L^\infty(\C)$. For instance, let $\mu_0(z):=\mz$ if $\Re(z)\in\bigcup_{n\in\Z}\brapa{n,n+\mz}$ and $\mu_0(z):=-\mz$ otherwise. Then the bi-Lipschitz homeomorphism $\psi_0:\C\to\C$ given by
		\begin{align*}
			&\psi_0(x+iy):=\begin{cases}
			n+\frac{9}{5}(x-n)+\frac{3}{5}iy=n+\frac{6}{5}(z-n)+\frac{3}{5}(\bar z-n) & n \le x \le n+\mz \\
			n+\frac{4}{5}+\frac{x-n}{5}+\frac{3}{5}iy=n+\frac{4}{5}+\frac{2}{5}(z-n)-\frac{1}{5}(\bar z-n) & n+\mz \le x \le n+1
			\end{cases}
		\end{align*}
		satisfies $\de_{\bar z}\psi_0=\mu_0\de_z\psi_0$, with the normalization $\psi_0(0)=0$ and $\psi_0(1)=1$. So $\mu_k:=\mu_0(2^{k}\cdot)$ and $\psi_k:=2^{-k}\psi_0(2^k\cdot)$ satisfy $\de_{\bar z}\psi_k=\mu_k\de_z\psi_k$ with the same normalization. Moreover, they converge uniformly to $\psi_\infty(x+iy)=x+\frac{3}{5}iy=\frac{4}{5}z+\frac{1}{5}\bar z$, together with their inverses. The homeomorphism $\psi_\infty$ satisfies $\de_{\bar z}\psi_\infty=\mu_\infty\de_z\psi_\infty$ with $\mu_\infty:=\frac{1}{4}$, but $\mu_k\weakstarto 0$. Dividing each $\psi_k$ by $\min_{\abs{z}=1}\abs{\psi_k(z)}$, we obtain a counterexample in the class $\mathcal{D}_{3}$.
	\end{rmk}

	\nocite{*}
	\bibliographystyle{plainnat}

\begin{thebibliography}{99}		
%
%
%
%
%
%
		\bibitem[Chodosh and Mantoulidis(2018)]{chma}
		O.~Chodosh and C.~Mantoulidis.
		\newblock Minimal surfaces and the Allen--Cahn equation on 3-manifolds: index, multiplicity, and curvature estimates.
		\newblock \emph{ArXiv preprint \href{https://arxiv.org/abs/1803.02716}{1803.02716}}, 2018.
		
%
%
%
%
%
%
%
		\bibitem[Ejiri and Micallef(2008)]{ejiri}
		N. Ejiri and M. Micallef.
		\newblock Comparison between second variation of area and second variation of energy of a minimal surface.
		\newblock \emph{Adv. Calc. Var.} 1(3):223--239,
		2008.
		\newblock \doi{10.1515/ACV.2008.009}.
		
%
%
%
%
%
%
%
%
		\bibitem[H{\'e}lein(2002)]{helein}
		F.~H{\'e}lein.
		\newblock \emph{Harmonic maps, conservation laws and moving frames (second edition)}, vol. 150 of \emph{Cambridge Tracts in Mathematics}.
		\newblock Cambridge University Press, Cambridge, 2002.
		
		\bibitem[Hummel(1997)]{hum}
		C.~Hummel.
		\newblock \emph{Gromov's compactness theorem for pseudo-holomorphic curves}, vol. 151 of \emph{Progress in Mathematics}.
		\newblock Birkh{\"a}user Verlag, Basel, 1997
	
		\bibitem[Imayoshi and Taniguchi(1992)]{imayoshi}
		Y.~Imayoshi and M.~Taniguchi.
		\newblock \emph{An introduction to {T}eichm\"uller spaces}.
		\newblock Springer--Verlag, Tokyo, 1992.
		
%
%
		\bibitem[Ketover and Liokumovich(2017)]{ketlio}
		D.~Ketover and Y.~Liokumovich.
		\newblock On the existence of unstable minimal Heegaard surfaces.
		\newblock \emph{ArXiv preprint \href{https://arxiv.org/abs/1709.09744}{1709.09744}}, 2016.

		\bibitem[Ketover and Marques and Neves(2016)]{ketmarnev}
		D.~Ketover, F.~C.~Marques and A.~Neves.
		\newblock The catenoid estimate and its geometric applications.
		\newblock \emph{ArXiv preprint \href{https://arxiv.org/abs/1601.04514}{1601.04514}}, 2016.

%
		\bibitem[Lehto and Virtanen(1973)]{lehto}
		O.~Lehto and K.~I.~Virtanen.
		\newblock \emph{Quasiconformal mappings in the plane (second edition)},
		vol. 126 of \emph{Grundlehren der mathematischen Wissenschaften}.
		\newblock Springer--Verlag, New York--Heidelberg--Berlin, 1973.
		
%
%
		\bibitem[Marques and Neves(2012)]{marnevrig}
		F.~C.~Marques and A.~Neves.
		\newblock Rigidity of min-max minimal spheres in three-manifolds.
		\newblock \emph{Duke Math. J.}, 161\penalty0 (14):\penalty0 2725--2752, 2012.
		\newblock \doi{10.1215/00127094-1813410}.

%
		\bibitem[Marques and Neves(2016)]{marnevmult}
		F.~C.~Marques and A.~Neves.
		\newblock Morse index and multiplicity of min-max minimal hypersurfaces.
		\newblock \emph{Camb. J. Math.}, 4\penalty0 (4):\penalty0 463--511, 2016.
		\newblock \doi{10.4310/CJM.2016.v4.n4.a2}.
		
		\bibitem[Michelat(2018)]{miche}
		A.~Michelat.
		\newblock On the Morse index of critical points in the viscosity method.
		\newblock \emph{ArXiv preprint \href{https://arxiv.org/abs/1806.09578}{1806.09578}}, 2018.

%
%
%
		\bibitem[Pigati and Rivi{\`e}re(2017)]{pigriv}
		A.~Pigati and T.~Rivi{\`e}re.
		\newblock The regularity of parametrized integer stationary varifolds in two dimensions.
		\newblock \emph{ArXiv preprint \href{https://arxiv.org/abs/1708.02211}{1708.02211}}, 2017.
		
%
%
%
%
		\bibitem[Rivi{\`e}re(2017)]{rivminmax}
		T.~Rivi{\`e}re.
		\newblock A viscosity method in the min-max theory of minimal surfaces.
		\newblock \emph{Publ. Math. Inst. Hautes \'Etudes Sci.}, 126:177–-246, 2017.
		\newblock \doi{10.1007/s10240-017-0094-z}.
		
		\bibitem[Rivi{\`e}re(2017)]{rivtarget}
		T.~Rivi{\`e}re.
		\newblock The regularity of conformal target harmonic maps.
		\newblock \emph{Calc. Var. Partial Differential Equations}, 56\penalty0 (4):117,
		2017.
		\newblock \doi{10.1007/s00526-017-1215-8}.
		
		\bibitem[Rivi{\`e}re(2018)]{rivlower}
		T.~Rivi{\`e}re.
		\newblock Lower semi-continuity of the index in the viscosity method for minimal surfaces.
		\newblock \emph{ArXiv preprint \href{https://arxiv.org/abs/1808.00426}{1808.00426}}, 2018.

		\bibitem[Rudin(1987)]{rudin}
		W.~Rudin.
		\newblock \emph{Real and complex analysis (third edition)}.
		\newblock McGraw--Hill Book Co., New York, 1987.
		
		\bibitem[Sacks and Uhlenbeck(1981)]{sacks}
		J.~Sacks and K.~Uhlenbeck.
		\newblock The existence of minimal immersions of {$2$}-spheres.
		\newblock \emph{Ann. of Math. (2)}, 113\penalty0 (1):\penalty0 1--24, 1981.
		\newblock \doi{10.2307/1971131}.
		
%
		\bibitem[Simon(1983)]{simon}
		L.~Simon.
		\newblock \emph{Lectures on geometric measure theory}, vol. 3 of \emph{Proceedings of the Centre for Mathematical Analysis}.
		\newblock Australian National University, Canberra, 1983.
		
%
		\bibitem[Song(2017)]{song}
		A.~Song.
		\newblock Local min-max surfaces and strongly irreducible minimal Heegaard splittings.
		\newblock \emph{ArXiv preprint \href{https://arxiv.org/abs/1706.01037}{1706.01037}}, 2017.

%
		\bibitem[Zhou(2015)]{zhouric}
		X.~Zhou.
		\newblock Min-max minimal hypersurface in {$(M^{n+1},g)$} with {$Ric>0$} and {$2 \leq n\leq 6$}.
		\newblock \emph{J. Differential Geom.}, 100\penalty0 (1):\penalty0 129--160, 2015.
		\newblock \doi{10.4310/jdg/1427202766}.
		
		\bibitem[Zhou(2017)]{zhou}
		X.~Zhou.
		\newblock On the existence of min-max minimal surface of genus {$g\geq 2$}.
		\newblock \emph{Commun. Contemp. Math.}, 19\penalty0 (4):\penalty0 1750041, 2017.
		\newblock \doi{10.1142/S0219199717500419}.
		
		\bibitem[Zhou(2019)]{zhoubumpy}
		X.~Zhou.
		\newblock On the multiplicity one conjecture in min-max theory.
		\newblock \emph{ArXiv preprint \href{https://arxiv.org/abs/1901.01173}{1901.01173}}, 2019.
	\end{thebibliography}

\end{document}